\documentclass{siamltex}
\usepackage[numbers]{natbib}
\usepackage{amsmath}
\usepackage{amscd}
\usepackage{amssymb}
\usepackage{graphicx}
\usepackage{subcaption}
\usepackage{url}
\usepackage{color}
\usepackage{listings}
\usepackage{booktabs}
\usepackage{caption}
\usepackage{textcomp}
\usepackage{mathtools}
\usepackage{mcode}
\usepackage{hyperref}
\definecolor{dark-gray}{gray}{0.3}
\definecolor{mauve}{rgb}{0.58,0,0.82}

\renewcommand{\vec}[1]{#1}
\newcommand{\Painleve}{Painlev\'e\ }

\definecolor{DarkBlue}{rgb}{0.00,0.00,0.55}
\definecolor{Black}{rgb}{0.00,0.00,0.00}
\hypersetup{
    linkcolor = DarkBlue,
    anchorcolor = DarkBlue,
    citecolor = DarkBlue,
    filecolor = DarkBlue,
    pagecolor = DarkBlue,
    urlcolor = DarkBlue,
    colorlinks  = true,
}

\title{Deflation techniques for finding distinct solutions of nonlinear partial differential equations}

\author{
  P. E. Farrell\thanks{Mathematical Institute, University of Oxford, Oxford, UK.
    Center for Biomedical Computing, Simula Research Laboratory, Oslo, Norway
    (\texttt{patrick.farrell@maths.ox.ac.uk}).}
  \and
  \'{A}. Birkisson\thanks{Mathematical Institute, University of Oxford, Oxford, UK
    (\texttt{birkisson@maths.ox.ac.uk}).
    }
  \and
  S. W. Funke\thanks{Center for Biomedical Computing, Simula Research Laboratory, Oslo, Norway and
    Department of Earth Science and Engineering,
    Imperial College London, London, UK.
    (\texttt{s.funke09@imperial.ac.uk}).
    This research is funded by EPSRC grant EP/K030930/1, an ARCHER RAP award, a Sloane Robinson Foundation Graduate Award associated with Lincoln College,
    by the European Research Council under the European Union's Seventh Framework Programme (FP7/2007-2013)/ERC grant 291068,
    and a Center of Excellence grant from the Research
    Council of Norway to the Center for Biomedical Computing at Simula
    Research Laboratory. The authors would like to acknowledge useful discussions with E.~S\"uli and L.~N.~Trefethen.
    }
  }

\begin{document}

\maketitle

\begin{abstract}
Nonlinear systems of partial differential equations (PDEs) may permit several distinct
solutions. The typical current approach to finding distinct solutions is to start Newton's method with
many different initial guesses, hoping to find starting points that lie in different basins of
attraction. In this paper, we present an infinite-dimensional deflation algorithm for systematically
modifying the residual of a nonlinear PDE problem to eliminate known solutions from consideration.
This enables the Newton--Kantorovitch iteration to converge to several different solutions, even starting from
the same initial guess. The deflated Jacobian is dense, but an efficient preconditioning strategy is
devised, and the number of Krylov iterations is observed not to grow as solutions are deflated. The power
of the approach is demonstrated on several problems from special functions, phase separation,
differential geometry and fluid mechanics that permit distinct solutions.
\end{abstract}



\section{Introduction}
Nonlinear problems may permit nontrivial distinct solutions. This paper is concerned with a
computational technique, called \emph{deflation}, for finding several distinct solutions of nonlinear (systems of)
partial differential equations.

Historically, the first application of related techniques was to finding distinct roots of scalar
polynomials \cite[pp. 78]{wilkinson1963}. Let $p(x)$ be a scalar polynomial, and let $x^{[0]},
x^{[1]}, \dots, x^{[n]}$ be roots of $p$ identified with some iterative algorithm, such as Newton's
method \cite{deuflhard2011}. Then further roots of $p$ may be found by considering the \emph{deflated} function
\begin{equation*}
q(x) = \frac{p(x)}{\displaystyle \prod_{i=0}^{n} (x - x^{[i]})},
\end{equation*}
and applying the same iterative algorithm to $q$.\footnote{In \cite{peters1971}, Peters and
Wilkinson draw a distinction between deflation (algebraically dividing the polynomial $p$ by $(x -
x^{[i]})$, and suppression (where the polynomial division is not performed explicitly, but the
numerator and denominator are separately evaluated and floating point division is performed after
the fact). In the subsequent literature, what these authors call suppression has come to be called
deflation, and we follow that convention here.}

Brown and Gearhart \cite{brown1971} extended this deflation approach to systems of nonlinear
algebraic equations, by considering \emph{deflation matrices} $M(x; r)$ that transform the
residual so that sequences that converge to a solution $r$ of the original problem will not converge to
that solution of the deflated problem.
Let $F$ be the residual of a system of nonlinear algebraic
equations, and let $r$ be a computed solution of $F(x) = 0$. Of the deflation matrices considered by
Brown and Gearhart, norm deflation extends most naturally, and is defined by choosing
\begin{equation*}
M(x; r) \equiv \frac{I}{\left|\left| x - r \right|\right|}
\end{equation*}
as its deflation operator, where $I$ is the appropriate identity matrix and $\left|\left|\cdot \right|\right|$ is some vector norm.
This yields the  modified residual function
\begin{equation*}
G(x) = M(x; r) F(x) = \frac{F(x)}{\left|\left| x - r\right|\right|}.
\end{equation*}
Brown and Gearhart prove that the deflated Newton sequence will not converge to the previous
solution, assuming $r$ is simple. (More precisely, a rootfinding algorithm that employs the
norm of the residual as its merit function will not converge to the previously-identified solution.)

This paper makes several novel contributions. First, we extend the theoretical framework of Brown
and Gearhart to the case of infinite-dimensional Banach spaces, enabling the application of
deflation techniques to systems of partial differential equations. Second, we introduce new
classes of deflation operators to overcome some numerical difficulties of previous methods. Third,
we discuss important details of applying these ideas in practice. Methods for solving PDEs typically
exploit the sparsity of the residual Jacobian; the deflated Jacobian is dense, but we devise
an efficient preconditioner for the dense deflated systems via the Sherman--Morrison formula.
Finally, we demonstrate its widespread applicability on several problems of interest
in the literature.

\subsection{Other related techniques}

There are two main alternative approaches to identifying distinct solutions of nonlinear systems:
numerical continuation, and the approximate integration of the associated Davidenko differential
equation.

The first approach, numerical continuation, is a well-established technique in the scientific
computing literature \cite{chao1975,chien1979, allgower1993, Book:AllgowerGeorg}.
The essential idea of it is as follows: suppose a problem $F$ with solution $u$ is parameterised by a parameter
$\lambda$:
\begin{equation}\label{eqn:parameterised}
 F(u, \lambda) = 0.
\end{equation}
Equation \eqref{eqn:parameterised} could represent an algebraic problem, or an operator equation
such as a PDE with boundary conditions. Respectively, $u$ will either be a vector in $\mathbb{R}^n$
or a function in some function space. The parameter $\lambda$ is usually a real scalar, but may be a
vector of parameters. For a fixed value $\lambda^*$, there may exist
zero, one, or many solutions $u$ for which $F(u, \lambda^*) = 0$. For some problems, the parameter
$\lambda$ appears naturally, whereas for others it may be artificially introduced (such as in
homotopy or incremental loading methods for solving difficult nonlinear equations). Studying how a
solution $u$ of \eqref{eqn:parameterised} varies with $\lambda$ is the subject of \emph{bifurcation
theory}.

Assume we have found one point $u^*$ for which $F(u^*, \lambda^*) = 0$. Then, following the implicit
function theorem in Banach spaces \cite[theorem\ 13.22]{Book:Hunter}, under technical conditions there exist open neighbourhoods
around $\lambda^*$ and $u^*$ and a unique function $f$ for which $u$ can be regarded as a function
of $\lambda$, that is, $u = f(\lambda)$, and $F(f(\lambda), \lambda) = 0$. It is thus possible to
define solution curves in the augmented solution space, which consist of points $(u, \lambda)$ for
which $F(u,\lambda) = 0$. Numerical continuation methods are concerned with tracing out these
curves, which give rise to bifurcation diagrams, an example of which is shown in figure
\ref{fig:bratu_bifurcation}.

\begin{figure}
\centering
\includegraphics[width=\textwidth]{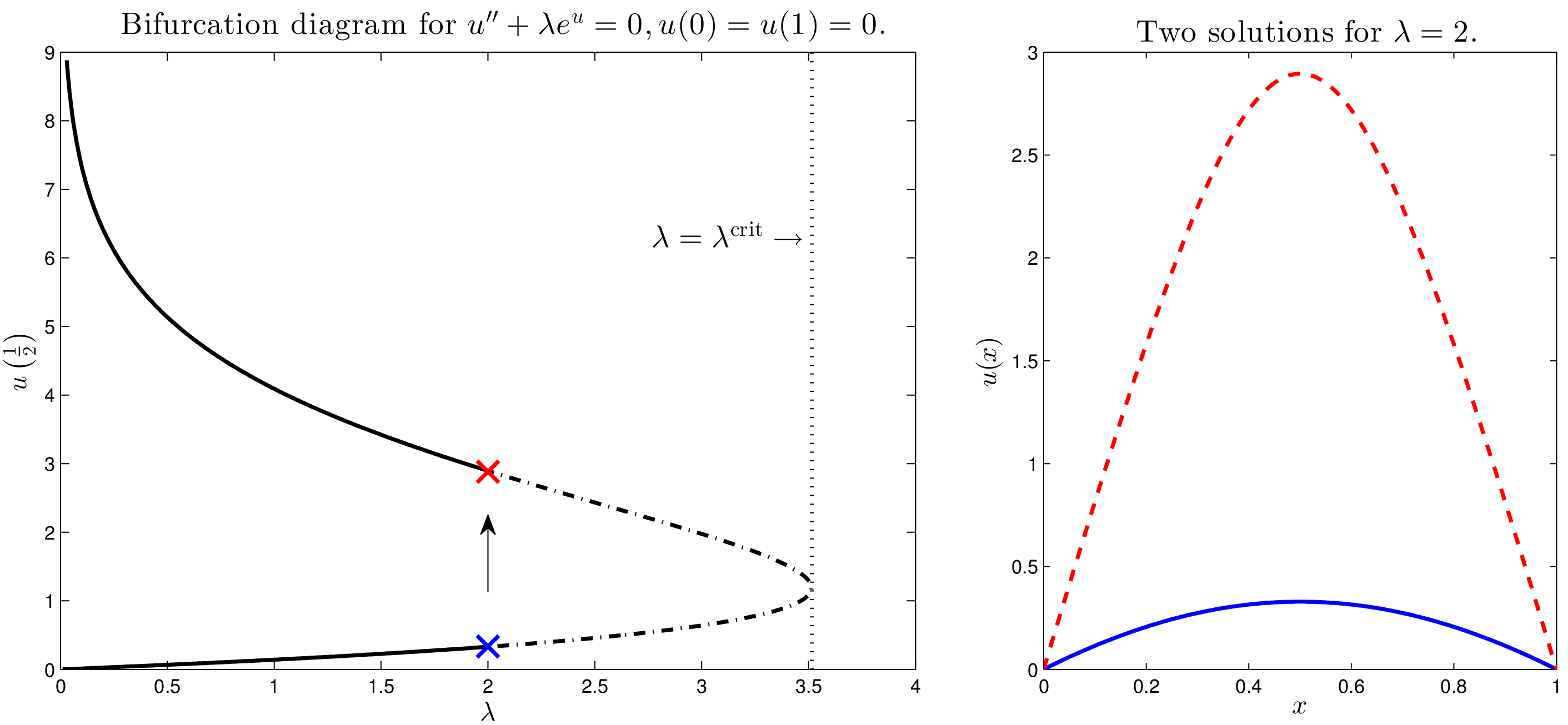}
\caption{A bifurcation diagram for the Bratu--Gelfand ODE \eqref{eqn:bratu}. On the right, the two distinct solutions for $\lambda = 2$ are plotted. For a fixed value of $\lambda$, deflation enables the rootfinding algorithm to find one solution after the other without changing $\lambda$, whereas numerical continuation traces the curve around the turning point (dashed-dotted line).}
\label{fig:bratu_bifurcation}
\end{figure}

While the solution curves can often be parameterised locally by $\lambda$, this parameterisation breaks down at points where the Fr\'echet derivative with respect to the solution
becomes singular. Thus, techniques such as arclength or pseudo-arclength continuation are often applied, so that
the continuation technique can extend beyond such points \cite{seydel2010}.

Consider the Bratu--Gelfand ordinary differential equation (ODE):
\begin{equation} \label{eqn:bratu}
\dfrac{\mathrm{d}^2 u}{\mathrm{d}x^2} + \lambda e^u = 0, \quad u(0) = u(1) = 0.
\end{equation}
For values of $\lambda < \lambda^{\textrm{crit}} \approx 3.51383$
equation \eqref{eqn:bratu} has two solutions; for $\lambda = \lambda^{\textrm{crit}}$, it has one solution,
and no solutions exist for $\lambda >  \lambda^{\textrm{crit}}$ (figure \ref{fig:bratu_bifurcation}) \cite{davis1962}. Fix a value $\lambda^* <
\lambda^{\textrm{crit}}$, and suppose one solution of the two is known. Numerical continuation
traces the solution curve around the turning point, increasing and decreasing $\lambda$ until the
second solution for $\lambda^*$ is identified. By contrast, deflation modifies the residual of the
problem to eliminate the first solution from consideration, enabling the Newton--Kantorovitch iteration
(henceforth referred to as Newton's method) to converge directly to the second solution without changing $\lambda$.

The second approach relies on the numerical integration of the Davidenko differential equation (DDE)
associated with the original nonlinear problem ${F}(u) = 0$ \cite{davidenko1953,branin1972}. The DDE introduces a new
arclength parameter $s$ and considers the augmented system
\begin{equation} \label{eqn:dde}
\dfrac{\mathrm{d}{F}(u(s))}{\mathrm{d}s} + {F}(u(s)) = 0,
\end{equation}
with initial condition $u(s=0)$ given by the initial guess to the solution. This has a strong
connection to Newton's method: provided the Fr\'echet derivative of $F$ with respect to $u$  is nonsingular, the chain
rule implies that
\begin{equation*}
\dfrac{\mathrm{d}{F}(u(s))}{\mathrm{d}s} = {F}'(u) \dfrac{\mathrm{d}u}{\mathrm{d}s}.
\end{equation*}
Hence \eqref{eqn:dde} can be rewritten as
\begin{equation} \label{eqn:dde_newton}
\dfrac{\mathrm{d}u}{\mathrm{d}s} = -\left({F}'(u)\right)^{-1} {F}(u),
\end{equation}
and thus Newton's method is just the forward Euler discretisation of \eqref{eqn:dde_newton}
with unit arclength step. Branin's method consists of considering the modified equation
\begin{equation*} \label{eqn:branin}
\dfrac{\mathrm{d}{F}(u(s))}{\mathrm{d}s} \pm {F}(u(s)) = 0,
\end{equation*}
where the sign is changed whenever the functional determinant of the Fr\'echet derivative changes sign or whenever a solution
is found. The major difficulty with implementing this method is that computing the determinant is
impractical for large-scale problems, where matrix decompositions are not feasible. Limited
computational experience indicates no performance benefit over deflation: applying the Newton-like
forward Euler discretisation of Branin's method to the Allen--Cahn problem of section
\ref{sec:allen_cahn} finds no solutions (whereas deflation finds three), and applying it to the Yamabe
problem of section \ref{sec:yamabe} finds two solutions (whereas deflation finds seven).

Finally, it is possible to employ techniques from numerical algebraic geometry to studying the
solutions of PDEs with polynomial nonlinearities; a comparison of the deflation technique
with this approach is given in section \ref{sec:hao}.

The deflation technique we present here is distinct from those algorithms (often also called
deflation) which aim to improve the convergence of Newton's method towards
\emph{multiple solutions}, solutions at which the Fr\'echet derivative is singular
\cite{ojika1983,griewank1985,leykin2006}. The algorithm presented here is also distinct from
deflation for eigenvalue problems.

\section{Deflation for PDEs}
\subsection{Deflation operators}
We now extend the results of \cite{brown1971} to the case of infinite-dimensional Banach spaces. This is the essential theoretical step in ensuring that deflation will apply to partial differential equations.

\begin{definition}[deflation operator on a Banach space \cite{birkisson2013}] \label{def:deflation}
Let $V,\, W$ and $Z$ be Banach spaces, and $U$ be an open subset of $V$. Let $\mathcal{F}\!\!:\!\!U \subset V \rightarrow W$ be a Fr\'echet
differentiable operator with derivative $\mathcal{F}'$. For each $r \in U$, $u \in U\setminus\{r\}$, let $\mathcal{M}(u; r)\!:\!W \rightarrow Z$ be an invertible linear operator. We say that $\mathcal{M}$ is a
\emph{deflation operator} if for any $\mathcal{F}$ such that $\mathcal{F}(r) = 0$ and
$\mathcal{F}'(r)$ is nonsingular, we have
\begin{align}
\liminf_{i \rightarrow \infty} \left|\left| \mathcal{M}(u_i; r) \mathcal{F}(u_i) \right|\right|_Z > 0
\end{align}
for any sequence $\left\{u_i\right\}$ converging to $r$, $u_i \in U\setminus\{r\}$.
\end{definition}

In order for solutions of the
deflated problem to yield information about solutions of the undeflated problem, the following
properties must hold:
\begin{enumerate}
\item (absence of spurious solutions). Suppose $\mathcal{M}(u; r) \mathcal{F}(u) = 0$ but 
$\mathcal{F}(u) \neq 0$. Then $\mathcal{M}(u; r)$ is a linear operator that maps both $0$ and $\mathcal{F}(u)$
to $0$; hence it is not invertible, and is not a deflation operator.

\item (preservation of other solutions). Suppose there exists $\tilde{r} \neq r$ such that $\mathcal{F}(\tilde{r}) = 0$.
Then $\mathcal{M}(\tilde{r}; r) \mathcal{F}(\tilde{r}) = 0$ by linearity of $\mathcal{M}(\tilde{r}; r)$.
\end{enumerate}

Informally, a deflation operator transforms the original equation $\mathcal{F}(u) = 0$ to
ensure that $r$ will not be found by any algorithm that uses the norm of the problem residual as
a merit function. Once a solution $r^{[0]}$ has been found (by any means), we form the new nonlinear
problem
\begin{align}
\mathcal{F}^{[1]}(u) \equiv \mathcal{M}(u; r^{[0]}) \mathcal{F}(u) = 0
\end{align}
and apply the rootfinding technique to $\mathcal{F}^{[1]}$. Clearly, this deflation procedure may be iterated
until the rootfinding technique diverges for some $\mathcal{F}^{[i]}$.

Brown and Gearhart introduced a \emph{deflation lemma} \cite[lemma 2.1]{brown1971} for determining
whether a matrix function can serve as a deflation matrix. We now extend this to deflation
operators on Banach spaces.

\begin{lemma}[sufficient condition for identifying deflation operators \cite{birkisson2013}] \label{lem:deflation}
Let $\mathcal{F}\!\!:\!\!U \subset V \rightarrow W$ be a Fr\'echet
differentiable operator. Suppose that the linear operator $\mathcal{M}(u; r)\!:\!W \rightarrow Z$ has the
property that for each $r \in U$, and any sequence $u_i~\overset{U}{\longrightarrow}~r, u_i \in U_r \equiv U\setminus\{r\}$, if
\begin{equation} \label{eqn:P}
\left| \left| u_i - r \right| \right| \mathcal{M}(u_i; r) w_i \overset{Z}\longrightarrow 0 \implies w_i \overset{W}\longrightarrow 0
\end{equation}
for any sequence $\left\{w_i\right\}, w_i \in W$, then $\mathcal{M}$
is a deflation operator.
\end{lemma}

\begin{proof}
Assume \eqref{eqn:P} holds.
If $\mathcal{M}$ is not a deflation operator, then there exists a Fr\'echet differentiable operator
$\mathcal{F}\!\!:\!\!U \subset V \rightarrow W$ and an $r \in U$ such that $\mathcal{F}(r) = 0$, $\mathcal{F}'(r)$ nonsingular and
\begin{equation*}
\liminf_{i \rightarrow \infty} \left|\left| \mathcal{M}(u_i; r) \mathcal{F}(u_i) \right|\right|_Z = 0
\end{equation*}
for some sequence $\left\{u_i\right\}$ converging to $r$, $u_i \in U_r$.
Then there exists a subsequence
$\left\{v_i\right\}$ such that $\mathcal{M}(v_i; r) \mathcal{F}(v_i) \overset{Z}\longrightarrow 0$. Defining $\left\{w_i\right\} \in W$ such that
\begin{equation*}
w_i = \frac{\mathcal{F}(v_i)}{\left|\left|v_i - r\right|\right|_U},
\end{equation*}
we have
\begin{equation*}
\left|\left|v_i - r\right|\right|_U \mathcal{M}(v_i; r) w_i \overset{Z}\longrightarrow 0.
\end{equation*}
By \eqref{eqn:P}, $w_i \overset{W}\longrightarrow 0$, i.e.
\begin{equation} \label{eqn:punchline}
\frac{\mathcal{F}(v_i)}{\left|\left|v_i - r\right|\right|_U} \overset{W}\longrightarrow 0.
\end{equation}
Since $\mathcal{F}$ is Fr\'echet differentiable, we can expand it in a Taylor series around $r$ to give
\begin{align*}
\mathcal{F}(v_i) &= \mathcal{F}(r) + \mathcal{F}'(r; v_i - r) + o(\left|\left|v_i - r\right|\right|^2_U) \\
                 &= \mathcal{F}'(r; v_i - r) + o(\left|\left|v_i - r\right|\right|^2_U),
\end{align*}
as $\mathcal{F}(r) = 0$. We then have that
\begin{align*}
\frac{\mathcal{F}(v_i)}{\left|\left|v_i - r\right|\right|_U} &= \frac{1}{\left|\left|v_i - r\right|\right|_U} \left[\mathcal{F}'(r; v_i - r) + o(\left|\left|v_i - r\right|\right|^2_U)\right] \\
 &\approx \mathcal{F}'(r; \bar{v_i}),
\end{align*}
where
\begin{equation*}
\bar{v_i} = \frac{(v_i - r)}{\left|\left|v_i - r\right|\right|_U} \in U_r
\end{equation*}
is a function with unit norm for all $v_i \in U_r$. But then \eqref{eqn:punchline} leads to a contradiction of the nonsingularity of $\mathcal{F}'$.
\end{proof}

The utility of this lemma is that it allows us to define candidate deflation operators $\mathcal{M}(u; r)$ and check whether property
\eqref{eqn:P} holds; if it holds, then $\mathcal{M}(u; r)$ is indeed a deflation operator. In the next section, we introduce several classes of deflation operators
to which lemma \ref{lem:deflation} applies.

\subsection{Classes of deflation operators} \label{sec:classes}
The simplest kind of deflation, Brown and Gearhart's norm deflation, extends naturally to the infinite-dimensional case.
\begin{definition}[norm deflation]
Norm deflation specifies
\begin{equation} \label{eqn:norm_deflation}
\mathcal{M}(u; r) = \frac{\mathcal{I}}{\left|\left|u - r\right|\right|_{{U}}},
\end{equation}
where $\mathcal{I}$ is the identity operator on $W$.
\end{definition}

That \eqref{eqn:norm_deflation} defines a deflation operator follows straightforwardly from lemma
\ref{lem:deflation}.

In computational practice it has occasionally proven useful to consider a generalisation of norm
deflation, where the norm is taken to a power:
\begin{definition}[exponentiated-norm deflation]
Exponentiated-norm deflation specifies
\begin{equation} \label{eqn:exponentiated_norm_deflation}
\mathcal{M}_p(u; r) = \frac{\mathcal{I}}{\left|\left|u - r\right|\right|^{{p}}_{{U}}},
\end{equation}
where $p \in \mathbb{R} \ge 1$.
\end{definition}

For example, in the Painlev\'e example of section \ref{sec:painleve}, the globalised
Newton algorithm fails to find the second solution when norm deflation is used, but succeeds
when squared-norm deflation $(p=2)$ is used. All examples below use either $p = 1$ or $p = 2$.

\begin{figure}
\centering
\includegraphics[width=\textwidth]{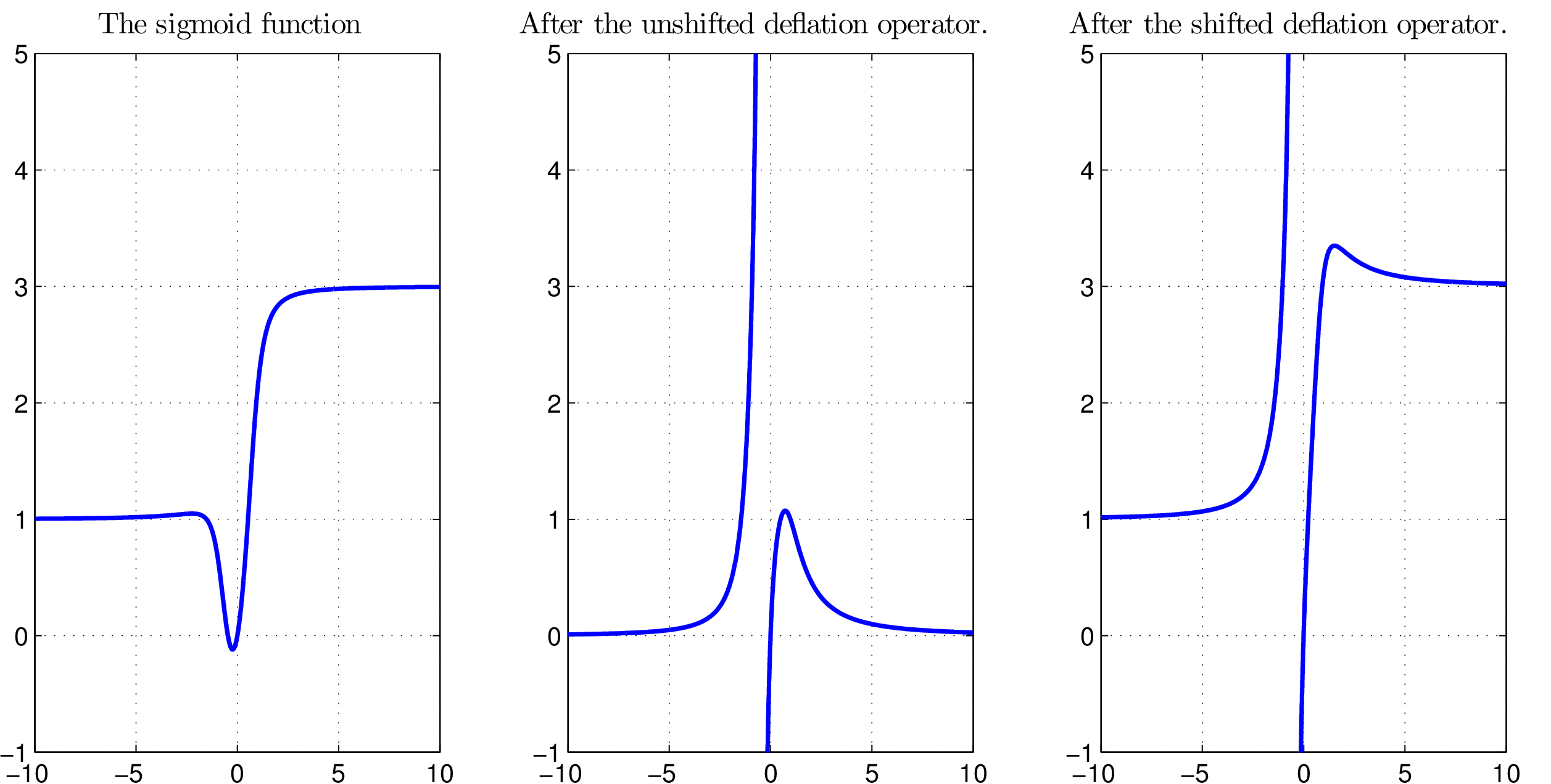}
\caption{Left: a plot of the sigmoid function \eqref{eqn:sigmoid} used to motivate the development
of shifted deflation. Centre: a plot of the function obtained after deflating the solution $x^{[0]}\approx-0.464$ with $p = 2$ as in \eqref{eqn:exponentiated_norm_deflation}.
The deflated function tends to zero away from the deflated solution, even where the original function did not. Right: a plot of the function obtained
after deflating the solution $x^{[0]}$ with $p = 2$, $\alpha = 1$ as in \eqref{eqn:shifted_deflation}. The deflated function tends to
the original function away from the deflated solution, due to the addition of the shift.}
\label{fig:sigmoid}
\end{figure}

While often successful, this class of deflation operator can sometimes induce numerical difficulties
for a nonlinear rootfinding algorithm: the rootfinding algorithm can erroneously report
convergence due to small residuals, even when $u$ is far away from a solution. This can be clearly seen
in the following example.  Consider the problem of finding the solutions of the sigmoid function
\begin{equation} \label{eqn:sigmoid}
    f(x) = \frac{x}{\sqrt{1 + x^2}} + \frac{2x^2}{\sqrt{1 + x^4}}.
\end{equation}
This function has two roots at $x=0$ and $x = -\sqrt{\frac{1}{3}\left(\sqrt{7} - 2\right)} \approx
-0.464$, and is plotted in figure \ref{fig:sigmoid} (left).
Starting from the
initial guess $x=-1$, undamped Newton iteration converges to the root $x \approx -0.464$ in 5 iterations. Suppose
exponentiated norm deflation is applied to this problem with $p=2$. When the Newton iteration is
applied for a second time starting from $x=-1$, the algorithm finds that it can make the norm of the
residual arbitrarily small by pushing $x$ towards $-\infty$: at $x \approx -1.2 \times 10^8$, the
deflated residual has norm on the order of $10^{-13}$, and the algorithm erroneously reports
successful convergence with a small residual (figure \ref{fig:sigmoid}, centre). While this example
is artificially constructed, similar erroneous behaviour has been observed in practice for more complicated problems,
such as the Allen--Cahn example of section \ref{sec:allen_cahn}.

This phenomenon of the deflation factor causing the residual to go to zero away from
the solution motivates the development of shifted deflation:
\begin{definition}[shifted deflation]
Shifted deflation specifies
\begin{equation} \label{eqn:shifted_deflation}
\mathcal{M}_{p, \alpha}(u; r) = \frac{\mathcal{I}}{\left|\left|u - r\right|\right|^{{p}}_{{U}}} + \alpha \mathcal{I},
\end{equation}
where $\alpha \ge 0$ is the shift.
\end{definition}

It follows from lemma \ref{lem:deflation} that $\mathcal{M}_{p, \alpha}(u; r)$ is also a
deflation operator.
The extra term $\alpha \mathcal{I}, \alpha > 0$ ensures that the norm of the deflated residual
does not artificially go to zero as $\left|\left|u - r\right|\right| \longrightarrow \infty$ (figure
\ref{fig:sigmoid}, right).
Instead, far away from previously found roots we have $\mathcal{M}_{p, \alpha}(u; r)
\mathcal{F}(u) \approx \alpha \mathcal{F}(u)$.
The observation that deflation scales the residual by $\alpha$ far away from deflated roots
immediately suggests a natural default value of $\alpha = 1$, as opposed to the
value of $\alpha = 0$ implicit in the deflation operators of Brown and
Gearhart. However, computational experience occasionally rewards choosing other
values of $\alpha$ (and scaling the termination criteria of the solver appropriately);
this is further discussed in the Yamabe example of section \ref{sec:yamabe}.

While the theory permits deflation operators built on invertible linear
operators other than the identity, we do not consider this further, as
all of the solution methods employed in this work are affine covariant.
Suppose we construct two deflation operators
\begin{equation}
\mathcal{M}_{\mathcal{I}} = \frac{\mathcal{I}}{\left|\left| u - r \right|\right|},
\end{equation}
and
\begin{equation}
\mathcal{M}_{\mathcal{A}} = \frac{\mathcal{A}}{\left|\left| u - r \right|\right|} = \mathcal{A} \mathcal{M}_{\mathcal{I}}
\end{equation}
for some invertible linear operator $\mathcal{A}$. Newton's method is \emph{affine covariant}:
considering a residual $\mathcal{A} \mathcal{F}(u)$ instead of $\mathcal{F}(u)$ does not change the
Newton iterates \cite{deuflhard2011}. Thus, deflation will yield the same iterates for both deflation
operators (up to roundoff), for all of the solution methods considered in this work.
However, it may be advantageous to consider other kinds of
deflation operators in the case where a non-affine-covariant solution algorithm
must be employed. For the remainder of this paper we only consider deflation
operators built on the identity.

\subsection{Summary of this section}
Definition \ref{def:deflation} describes the property that we
demand of a deflation operator that works for PDEs; lemma \ref{lem:deflation}
gives a sufficient condition for being a deflation operator that is easy to
test; and section \ref{sec:classes} introduces several kinds of deflation
operator to which lemma \ref{lem:deflation} applies. 

However, none of this would be practically useful for finding distinct solutions
if the key computational kernel of Newton's method, the solution of a linear
system involving the Jacobian arising from some discretisation, cannot be made
to scale up to large problems in the presence of deflation. We now turn our
attention from theoretical concerns to the details required to successfully
implement deflation for large-scale discretised problems.

\section{Implementation: sparsity and preconditioning} \label{sec:preconditioning}
Exploiting sparsity of the (discretised) Jacobian in the Newton step is critical for computational
efficiency for many discretisation techniques for PDEs, such as finite differences and finite elements.
Suppose the application of deflation yields
\begin{equation*}
\mathcal{G}(u) = \frac{\mathcal{F}(u)}{\eta(u)},
\end{equation*}
where the exact form of $\eta:U \mapsto \mathbb{R}$ depends on the kind of deflation operator employed and the
number of solutions deflated. ($\eta(u)$ may include multiple deflations.) The action of the Fr\'echet derivative
of $\mathcal{G}$ in a direction $\delta u$ is given by
\begin{equation*}
\mathcal{G}'(u; \delta u) = \frac{\mathcal{F}'(u; \delta u)}{\eta(u)} - \frac{\mathcal{F}(u)}{\eta(u)^2} \eta'(u; \delta u).
\end{equation*}

Suppose the problem is discretised so that the solution $\vec{u}$ is sought in $\mathbb{R}^N$, where
$N$ is the number of degrees of freedom.  The discretisation $J_G \in \mathbb{R}^{N \times N}$ of
$\mathcal{G}'(u)$ will be dense even if the discretisation $J_F \in \mathbb{R}^{N \times N}$ of
$\mathcal{F}'(u)$ is sparse, as $J_G$ is a rank-one perturbation of $J_F$:
\begin{equation} \label{eqn:deflated_jacobian}
J_G(u) = \frac{J_F(\vec{u})}{\eta(\vec{u})} - \frac{\vec{F}(u)}{\eta(\vec{u})^2} \otimes \vec{d(u)}
\end{equation}
where $\vec{F}(u) \in \mathbb{R}^N$ is the discretisation of $\mathcal{F}(u)$ and $\vec{d}(u) \in
\mathbb{R}^N$ is the discretisation of $\eta'(u; \cdot)$.

This has two practical implications for a software implementation. First, the discretised deflated Jacobian
$J_G$ should not be stored explicitly, but rather its action should be computed matrix-free via
\eqref{eqn:deflated_jacobian}. Second, some care must be taken in preconditioning linear systems
involving $J_G$ (when Newton--Krylov methods are used). Suppose that a left preconditioner $P_F$ is
available for the matrix $J_F$, i.e.\@ that the undeflated Newton--Krylov method approximately solves
\begin{equation} \label{eqn:undeflated_newton_system}
P_F^{-1} J_F \vec{x} = P_F^{-1} \vec{b}
\end{equation}
using some Krylov method. Neglecting dependence on $u$ for notational brevity, this suggests using
the preconditioner
\begin{equation}
P_G = \frac{P_F}{\eta} - \frac{\vec{F}}{\eta^2} \otimes \vec{d},
\end{equation}
for the deflated linear system:
\begin{equation} \label{eqn:deflated_newton_system}
P_G^{-1} J_G \vec{x} = P_G^{-1} \vec{b}.
\end{equation}
The action of $P_G^{-1}$ may be computed via the Sherman--Morrison formula
\cite{bartlett1951,hager1989}:
\begin{equation*}
(\frac{P_F}{\eta} + \vec{f}\vec{d}^T)^{-1} = \eta P_F^{-1} - \eta^2 \frac{P_F^{-1} \vec{f} \vec{d}^T P_F^{-1}}{1 + \eta \vec{d}^T P_F^{-1} \vec{f}},
\end{equation*}
where $\vec{f} = -{\vec{F}}/{\eta^2}$, provided $1 + \eta \vec{d}^T P_F^{-1} \vec{f} \ne 0$. This
allows for the matrix-free computation of actions of $P_G^{-1}$ in terms of actions of $P_F^{-1}$.

We hope that if $P_F$ is a useful preconditioner for $J_F$, then $P_G$ will be a useful preconditioner
for $J_G$. Expanding, we find
\begin{align*}
P_G^{-1} J_G &= (\frac{P_F}{\eta} + \vec{f}\vec{d}^T)^{-1} (\frac{J_F}{\eta} + \vec{f}\vec{d}^T) \\
             &= (\eta P_F^{-1} - \eta^2 \frac{P_F^{-1} \vec{f} \vec{d}^T P_F^{-1}}{1 + \eta \vec{d}^T P_F^{-1} \vec{f}}) (\frac{J_F}{\eta} + \vec{f}\vec{d}^T) \\
             &= P_F^{-1} J_F + \eta P_F^{-1} \vec{f}\vec{d}^T - \frac{\eta P_F^{-1} \vec{f} \vec{d}^T P_F^{-1}J_F + \eta^2 P_F^{-1} \vec{f} \vec{d}^T P_F^{-1}\vec{f}\vec{d}^T}{1 + \eta \vec{d}^T P_F^{-1} \vec{f}} \\
             &= P_F^{-1} J_F + \eta P_F^{-1} \vec{f}\vec{d}^T - \frac{\eta P_F^{-1} \vec{f} \vec{d}^T + \eta^2 P_F^{-1} \vec{f} \vec{d}^T P_F^{-1}\vec{f}\vec{d}^T}{1 + \eta \vec{d}^T P_F^{-1} \vec{f}}
                                                         - \frac{\eta P_F^{-1} \vec{f} \vec{d}^T\left(P_F^{-1} J_F - I\right)}{1 + \eta \vec{d}^T P_F^{-1} \vec{f}} \\
             &= P_F^{-1} J_F + \eta P_F^{-1} \vec{f}\vec{d}^T - \frac{\eta P_F^{-1} \vec{f}\left(1 + \eta \vec{d}^T P_F^{-1} f\right) \vec{d}^T}{1 + \eta \vec{d}^T P_F^{-1} \vec{f}}
                                                         - \frac{\eta P_F^{-1} \vec{f} \vec{d}^T\left(P_F^{-1} J_F - I\right)}{1 + \eta \vec{d}^T P_F^{-1} \vec{f}} \\
             &= P_F^{-1} J_F - \frac{\eta P_F^{-1} \vec{f} \vec{d}^T\left(P_F^{-1} J_F - I\right)}{1 + \eta \vec{d}^T P_F^{-1} \vec{f}}.
\end{align*}
This expression has two consequences. First, if $P_F^{-1} = J_F^{-1}$ (computed with an LU or Cholesky decomposition) so that
\eqref{eqn:undeflated_newton_system} is solved in one iteration, then
\eqref{eqn:deflated_newton_system} will also be solved in one iteration, as the error term is multiplied by $P_F^{-1} J_F - I$.
Second, on taking matrix norms, we find
\begin{align*}
\left|\left|P_G^{-1} J_G - P_F^{-1} J_F\right|\right| &=   \left|\frac{\eta}{1 + \eta \vec{d}^T P_F^{-1} \vec{f}}\right| \left|\left|{P_F^{-1} \vec{f} \vec{d}^T \left(P_F^{-1} J_F - I\right)}\right|\right| \\
                                                      &\le \left|\frac{\eta}{1 + \eta \vec{d}^T P_F^{-1} \vec{f}}\right| \left|\left|P_F^{-1} \vec{f} \vec{d}^T\right|\right| \left|\left|P_F^{-1} J_F - I\right|\right|.
\end{align*}
We are mainly interested in the performance of the deflated preconditioner
away from previous roots, as that is where the deflated Newton iteration will
take us. Let us examine how each term scales with large $\eta$ ($\eta \approx \alpha^{-1}$ may be large
if $\alpha \ll 1$). The vector $f$ scales with $\eta^{-2}$, so the denominator of the first term is
approximately 1, and the first term scales linearly with $\eta$. The second term scales with $\eta^{-2}$,
and the third term is independent of $\eta$. Therefore, for large $\eta$, we expect the 
difference between the deflated and undeflated preconditioned operators will be
small if $P_F$ is a good preconditioner for $J_F$ in the sense that
$\left|\left|P_F^{-1} J_F - I\right|\right|$ is small. 

Computational experience with this preconditioned Newton--Krylov method is
reported in the examples of sections \ref{sec:allen_cahn} and \ref{sec:yamabe}.
In practice, the number of Krylov iterations required is observed not to grow
as more and more solutions are deflated.

\section{Examples}
\subsection{Special functions: Painlev\'e} \label{sec:painleve}

\begin{figure}
\centering
\includegraphics[width=\textwidth]{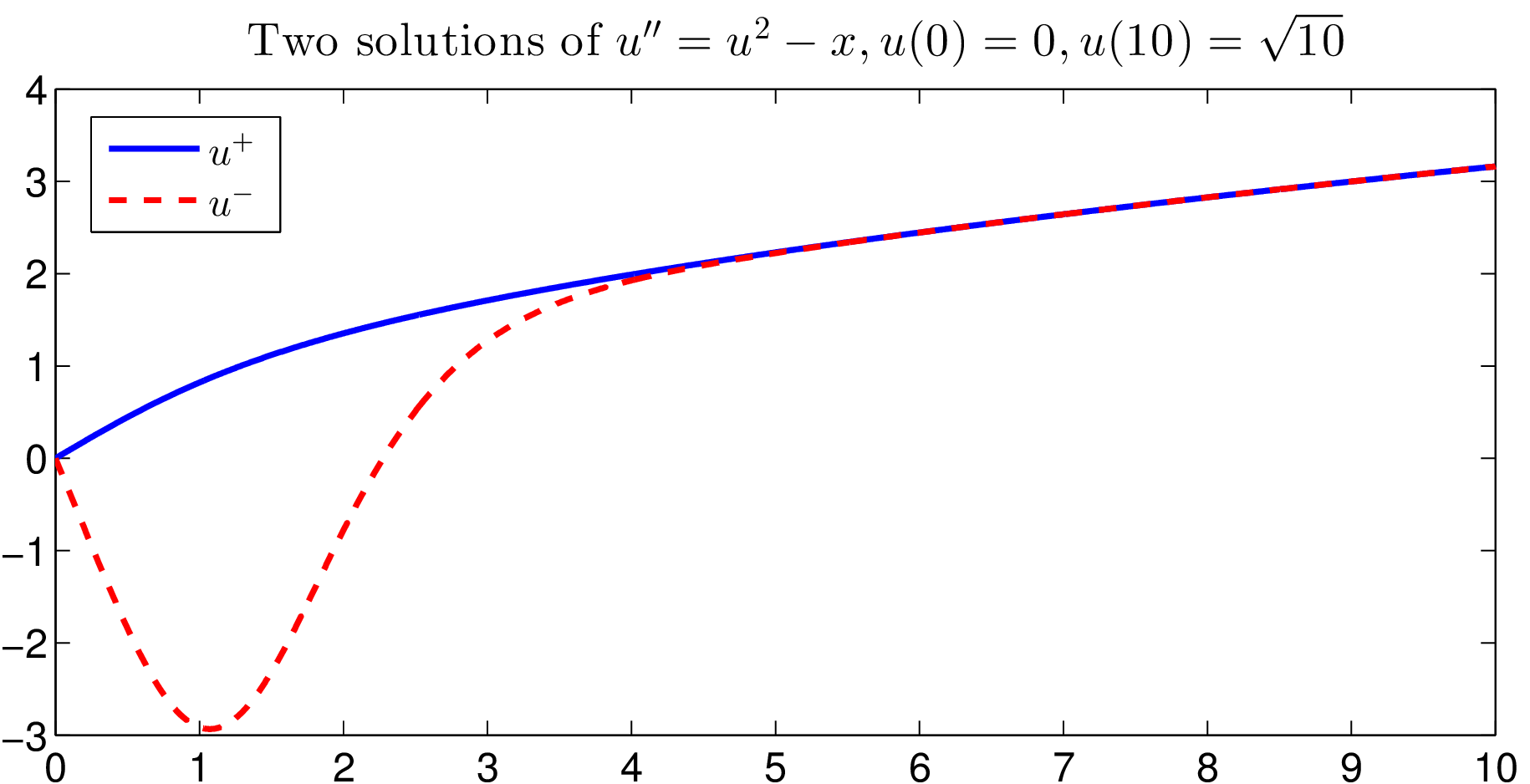}
\caption{The positive- and negative-slope solutions $u^+$ and $u^-$ of the \Painleve boundary value problem \eqref{eqn:painleveTrunc}.}
\label{fig:painleveheavenhell}
\end{figure}

A well-studied example of an ordinary differential equation boundary-value problem that permits
distinct solutions is based on the first \Painleve transcendent \cite{Holmes:1984:PTB,
Noonburg:1995:SSP, Fornberg:2011:NMP}, and is given by
\begin{equation}\label{eqn:painleveAB}
\dfrac{\mathrm{d}^2u}{\mathrm{d}x^2} = u^2 - x, \quad u(0) = 0, \quad u(x) \sim \sqrt{x} \text{ as } x \rightarrow \infty,
\end{equation}
where the latter condition means that the solution should asymptote to $\sqrt{x}$.
In \cite{Art:HastingsTroy:1989}, it was shown that exactly two solutions exist that satisfy these
conditions, one with a positive slope at $x=0$, and the other with a negative slope at $x=0$. We refer to these solutions as $u^+$ and $u^-$.
The first solution is easy to compute, while computing the second is far more computationally challenging: almost all convergent initial guesses attempted converge to $u^+$. Here, we truncate \eqref{eqn:painleveAB} to the interval
$[0, 10]$, and consider the task of finding the two solutions to the boundary-value problem:
\begin{equation}\label{eqn:painleveTrunc}
\dfrac{\mathrm{d}^2u}{\mathrm{d}x^2} = u^2 - x, \quad u(0) = 0, \quad u(10) = \sqrt{10}.
\end{equation}

Using the \verb|chebop| class of Chebfun \cite{Birkisson2012, Book:ChebfunGuide}, the $u^+$ solution can be obtained via spectral methods \cite{Book:Trefethen}, using the affine-covariant
globalised \texttt{NLEQ-ERR} Newton algorithm of \cite[pg. 148]{deuflhard2011} in function space.
All linear systems arising were solved with the \texttt{lu} function of \textsc{Matlab}.
The initial guess for the Newton iteration was set to the linear function that satisfies the boundary
conditions, and solved with the following lines of \textsc{Matlab} code:
\begin{lstlisting}
% Define the interval to solve the problem on:
domain = [0, 10];
% Construct a chebop object, representing the differential operator:
N = chebop(@(x,u) diff(u, 2) - u.^2 + x, domain);
% Impose Dirichlet boundary conditions:
N.lbc = 0; N.rbc = sqrt(10);
% Solve using overloaded \ method:
uplus = N\0;
\end{lstlisting}

No damping was required for the Newton iteration to converge from the initial guess used. Deflation was then applied with $p = 2$ and $\alpha = 0$, and the $u^-$ solution was computed starting from the same initial guess. Globalised damped Newton iteration was applied to achieve convergence to the $u^-$ solution,
requiring 12 steps to converge. (Deflation with $p=1$ failed to converge, even with globalisation.)
In this case, globalisation was necessary: applying the undamped exact Newton iteration to the deflated problem
diverged for every value of $p$ and $\alpha$ attempted.  The two solutions obtained are plotted in
figure \ref{fig:painleveheavenhell}.

\subsection{Comparison with a domain decomposition appoach: Hao et  al.} \label{sec:hao}
Hao et al.\ \cite{hao2014} have recently proposed an entirely different algorithm
for finding multiple solutions of PDEs based on domain decomposition and the
application of techniques from numerical algebraic geometry. The algorithm is
restricted to differential equations with polynomial nonlinearities: it solves
the polynomial systems arising from subdomains with sophisticated numerical algebraic
geometry software \cite{bates2013}, and then employs continuation in an attempt to
find solutions of the original discretised system. The authors do not claim that their
approach finds all solutions. In this section, we consider
example 1 of \cite{hao2014}, and compare the efficiency of deflation against
the algorithm presented there.

The problem to be solved is
\begin{equation} \label{eqn:hao2014}
u_{xx} = - \lambda (1 + u^4) \quad \textrm{ on } \Omega = (0, 1),
\end{equation}
with a homogeneous Neumann condition on the left and a homogeneous Dirichlet
boundary condition on the right. Hao et al.\ showed that for $0 < \lambda <
\lambda^* \approx 1.30107$, this problem has two solutions; at $\lambda =
\lambda^*$, it has one solution, and above $\lambda^*$ it has no solutions.

The equation was discretised with piecewise linear finite elements, yielding a
nonlinear system with $\mathcal{O}(100)$ degrees of freedom. The nonlinear
system was solved with undamped Newton iteration, and the arising linear systems
were solved with MUMPS. Deflation ($p = 1$, $\alpha = 1$) was applied to find
both solutions, starting from the same initial guess of zero. The deflation
approach found both solutions for $\lambda = 1.2$ in approximately one second on
a single core of a laptop computer. By contrast, the algorithm presented by Hao
et al.\ applied to a discretisation with $\mathcal{O}(100)$ degrees of freedom
took approximately 5 hours and 39 minutes on 96 cores \cite[table 1]{hao2014}.
While the approach of Hao at al.\ may find solutions that deflation does not
(and vice versa), its cost appears to be prohibitive.

\subsection{Phase separation: Allen--Cahn} \label{sec:allen_cahn}
The Allen--Cahn equation \cite{allen1979} was proposed to model the motion
of boundaries between phases in alloys, and is well-known to permit several
steady solutions.

The equation considered is the steady Allen--Cahn equation:
\begin{align} \label{eqn:ac}
-\delta \nabla^2 u + {\delta}^{-1}(u^3 - u) = 0,
\end{align}
where $u = +1$ corresponds to one material, and $u = -1$ the other, and $\delta = 0.04$ is a
parameter relating the strength of the free surface tension to the potential term in the
free energy. The equation is solved on the unit square $\Omega = (0, 1)^2$ with boundary conditions:
\begin{align} \label{eqn:acbcs}
\begin{split}
u &= +1 \quad \mathrm{on} \ x = 0, \ x = 1, \ 0 < y < 1, \\
u &= -1 \quad \mathrm{on} \ y = 0, \ y = 1.
\end{split}
\end{align}

The equation was discretised with piecewise linear finite elements on a $100 \times 100$ mesh via
FEniCS \cite{logg2011}.  Undamped Newton iteration was applied to solve the nonlinear problem.

This example was studied by \cite{e2004} in the context of identifying minimal action pathways and
transition times between stable states. In the absence of noise, a system evolves to one of its
stable states and stays there indefinitely; by contrast, when stochastic noise is present, the
system may switch between metastable states when the noise perturbs the system out of the basin of
attraction for its steady solution. The string method of E, Ren and Vanden--Eijnden
\cite{e2002,e2004,e2007} attempts to identify the most probable transition path between two
metastable states, and hence to characterise the transition timescales. However, this requires
some foreknowledge of the steady solutions, so that the endpoints of the string may be initialised
appropriately. Deflation may be used to identify the different solutions of the associated steady
problem.

\begin{figure}
\centering
  \begin{subfigure}[b]{0.3\textwidth}
    \includegraphics[width=\textwidth]{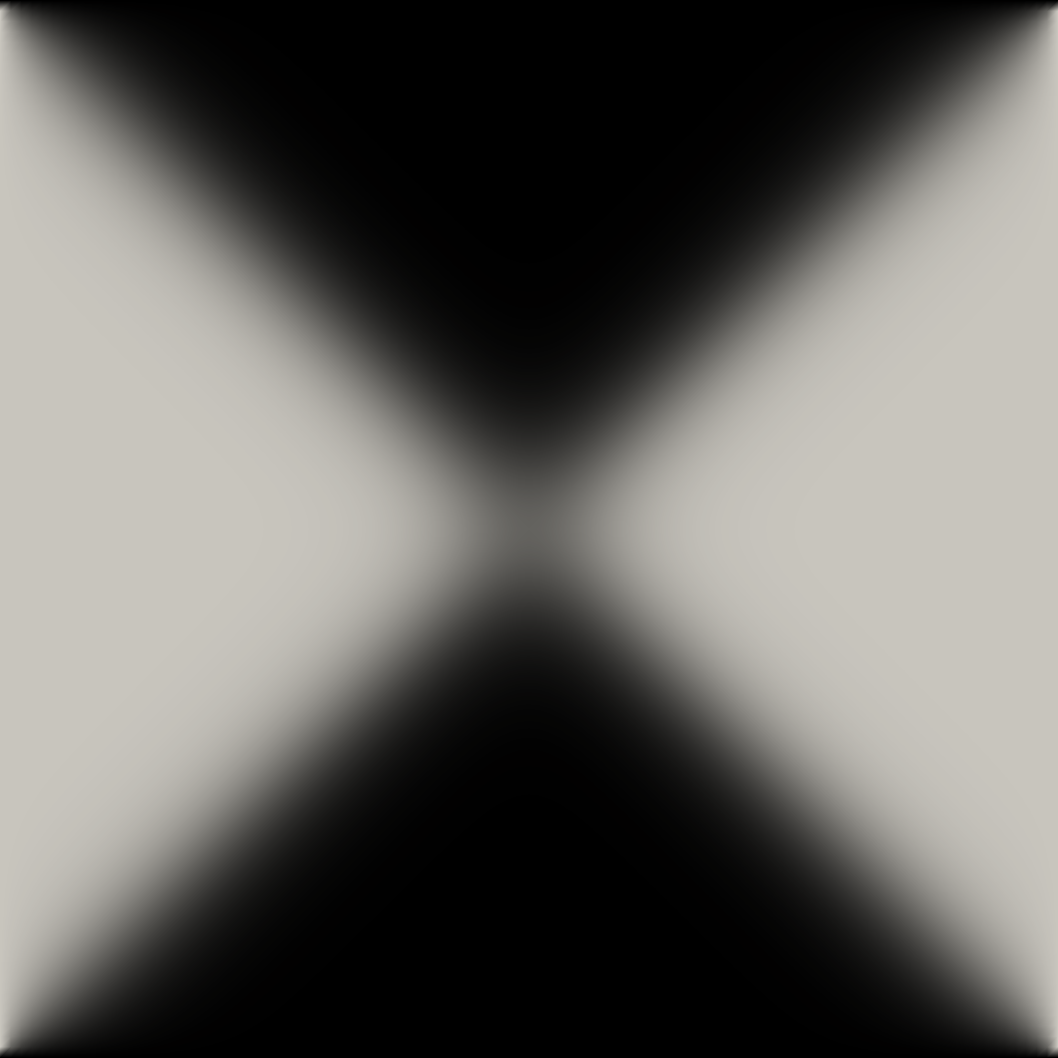}
    \caption{Unstable solution}
    \label{fig:allen_cahn_0}
  \end{subfigure}
  \begin{subfigure}[b]{0.3\textwidth}
    \includegraphics[width=\textwidth]{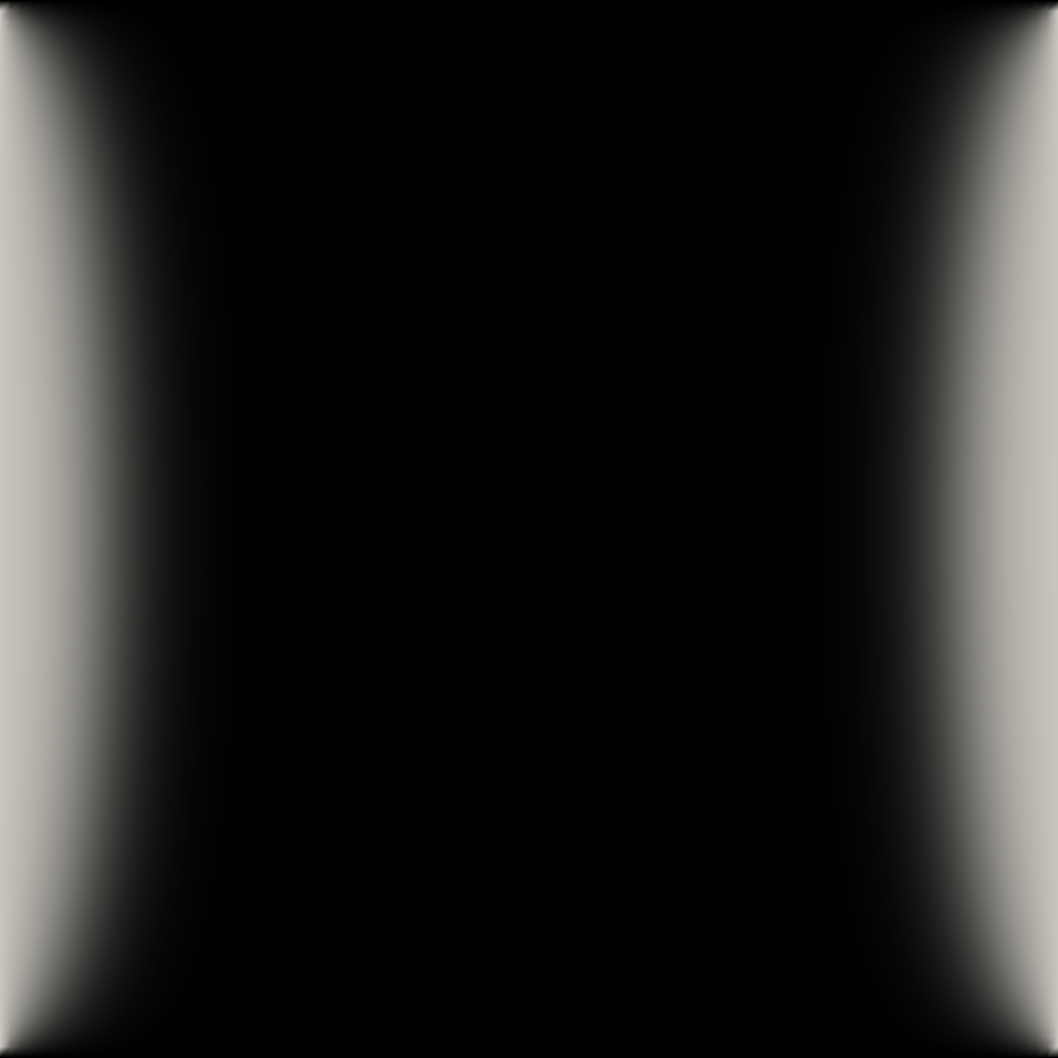}
    \caption{Negative stable solution}
    \label{fig:allen_cahn_1}
  \end{subfigure}
  \begin{subfigure}[b]{0.3\textwidth}
    \includegraphics[width=\textwidth]{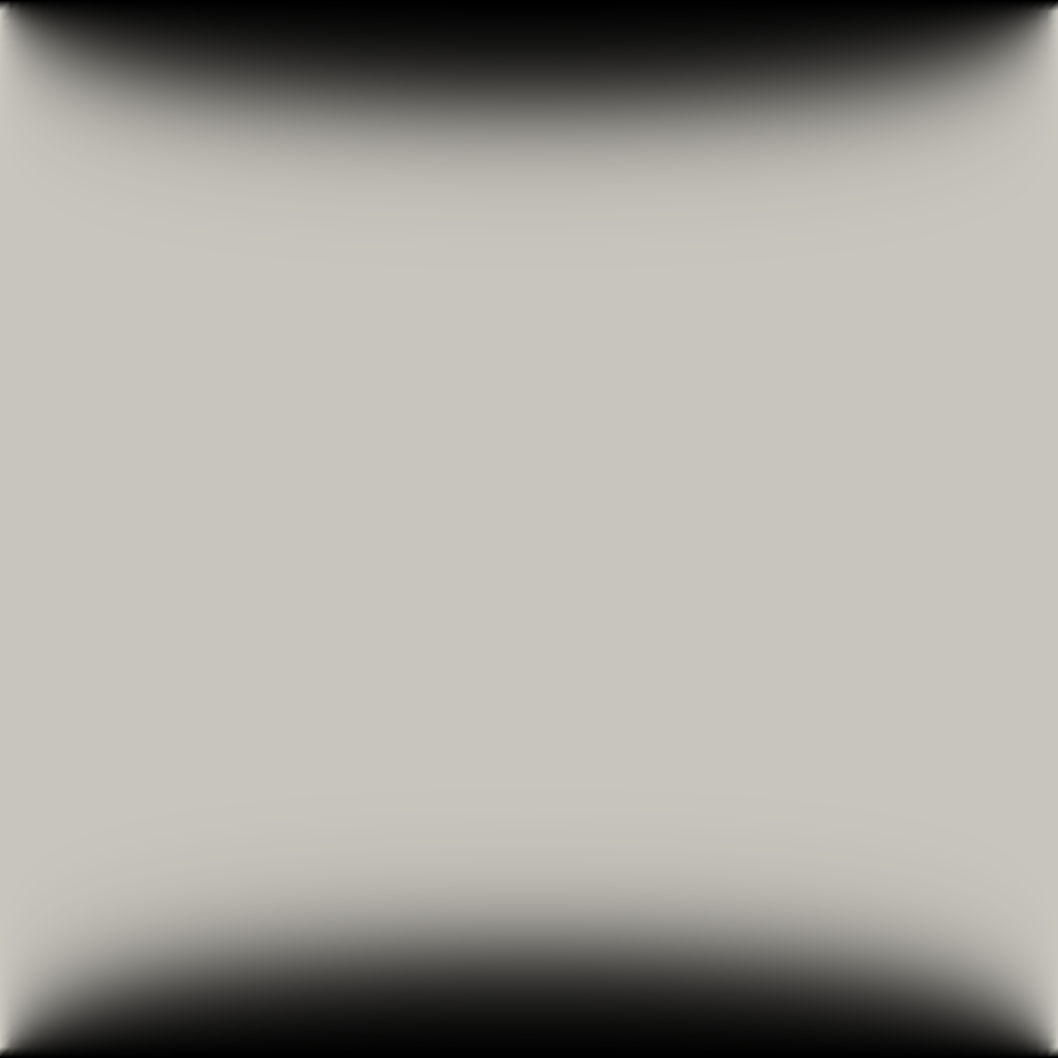}
    \caption{Positive stable solution}
    \label{fig:allen_cahn_2}
  \end{subfigure}
  \caption{Solutions of the Allen--Cahn equation \eqref{eqn:ac}--\eqref{eqn:acbcs}; c.f. figure 4.1 of \cite{e2004}, left column.}
  \label{fig:allen_cahn}
\end{figure}

\begin{table}
\centering
\begin{tabular}{c|c}
\toprule
\# of deflations & average Krylov iterations per solve \\
\midrule
0 & 10.84 \\
1 & 10.58 \\
2 & 10.53 \\
\bottomrule
\end{tabular}
\caption{The performance of the preconditioning strategy of section \ref{sec:preconditioning}
on the Allen--Cahn problem, section \ref{sec:allen_cahn}. As more solutions are deflated, the
number of Krylov iterations required does not increase, indicating that the
preconditioning strategy suggested is effective.}
\label{tab:allen_cahn_preconditioning}
\end{table}

Starting each time from an initial guess of zero, Newton's method with deflation ($p=1$, $\alpha=0$) finds
three solutions (figure \ref{fig:allen_cahn}): one unstable solution (figure \ref{fig:allen_cahn_0}), and two
stable solutions (figures \ref{fig:allen_cahn_1}, \ref{fig:allen_cahn_2}). This information about the
stable solutions would be sufficient to initialise the string method to find the minimum energy pathway
(which, in this case, happens to pass through the unstable solution).

The linear systems arising in the Newton algorithm were solved using GMRES \cite{saad1986} and the
GAMG classical algebraic multigrid algorithm \cite{adams2004} via PETSc \cite{balay1997}, with two
smoother iterations of Chebyshev and SOR. The relative and absolute Krylov solver tolerances were
both set to $10^{-12}$.  The average number of Krylov iterations required in the Newton iterations
for each solution are listed in table \ref{tab:allen_cahn_preconditioning}. The number of Krylov iterations
required for the deflated solves stays approximately constant. This suggests that the
preconditioning strategy proposed in section \ref{sec:preconditioning} is effective, even as
several solutions are deflated.

This example was further used to investigate the convergence of Newton's method
for different values of the deflation parameters. Of course, many factors
influence whether Newton's method will converge, even if a solution exists, including:
\begin{itemize}
\item the
proximity of the initial guess; 
\item the globalisation technique employed (e.g. $L^2$, critical point, backtracking or \texttt{NLEQ-ERR} linesearch); 
\item the use of Jacobian lagging \cite{ortega2000,brown2013}; 
\item the use of nonlinear preconditioning techniques \cite{brune2013}.
\end{itemize}
This makes a systematic search through the design space
of deflation algorithms both intractable and somewhat irrelevant (for the
results will differ completely for a different nonlinear problem, or even a
different initial guess).  Nevertheless, we remove as many degrees of freedom as
possible by fixing the algorithm as undamped unlagged Newton's method with exact
linear solves (up to roundoff) and 100 maximum iterations, and report the number
of roots found while varying $p$ and $\alpha$.

\begin{table}
\centering
\begin{tabular}{c|c}
\toprule
configuration $(p, \alpha)$ & number of solutions found \\
\midrule
$(1, 0)$ & 3 \\
$(1, 0.1)$ & 3 \\
$(1, 1)$ & 2 \\
$(2, 0)$ & $2^*$ \\
$(2, 0.1)$ & 2 \\
$(2, 1)$ & 3 \\
\bottomrule
\end{tabular}
\caption{Number of solutions found for various deflation parameters for the
Allen--Cahn problem of section \ref{sec:allen_cahn}. The $^*$ for $(2, 0)$
indicates that Newton's method erroneously reported convergence to a third
root, but it was spurious for the reasons illustrated in figure
\ref{fig:sigmoid}.}
\label{tab:allen_cahn_solution_count}
\end{table}

The results are shown in table \ref{tab:allen_cahn_solution_count}. For each
value of $p$ considered, there is a shift with which it finds 2 or 3 solutions;
similarly, for each shift considered there is a $p$ with which it finds 2 or 3
solutions. It seems difficult to predict in advance which values of $(p,
\alpha)$ will find all solutions. However, the deflation approach
finds more than 1 solution for all cases considered, indicating its partial success.

\subsection{Differential geometry: Yamabe} \label{sec:yamabe}
In 1960, Hidehiko Yamabe \cite{yamabe1960} posed the following problem: given a compact manifold $M$ of dimension $n \ge 3$
with Riemannian metric $g$, is it possible to find a metric $\tilde{g}$ conformal to $g$ (a multiplication of $g$
by a positive function) that has constant scalar curvature? Yamabe showed that solving this problem
is equivalent to finding a $u$ such that
\begin{equation} \label{eqn:yamabe_generic}
-a \nabla^2 u - S u + \lambda u^{p-1} = 0,
\end{equation}
where $a = 4(n-1)/(n-2)$, $p = 2n/(n-2)$, $S$ is the scalar curvature of $g$, and $\lambda$ is the constant
scalar curvature of $\tilde{g}$. This problem has a solution; for a full account, see \cite{lee1987}. This is
an instance of a more general class of critical exponent problems \cite{brezis1983,erway2011}, which often permit
distinct solutions.

Following \cite{erway2011}, we consider a Yamabe-like problem arising from $n=3$, but posed on
a two-dimensional domain $\Omega$:
\begin{equation} \label{eqn:yamabe}
-8 \nabla^2 u - \frac{1}{10} u + \frac{1}{r^3} u^5 = 0,
\end{equation}
where $\Omega$ is an annulus centred on the origin of inner radius 1 and outer radius 100, and $r$
is the distance to the origin. The system is closed by imposing that $u = 1$ on the boundaries. In
\cite{erway2011}, the authors are concerned with devising schemes to solve the nonlinear problem
\eqref{eqn:yamabe} subject to the constraint that $u \ge 0$; here, the nonnegativity constraint
is relaxed, and solutions of any sign are sought using deflation.

\begin{figure}
\begin{tabular}{ccc}
\includegraphics[width=3.9cm]{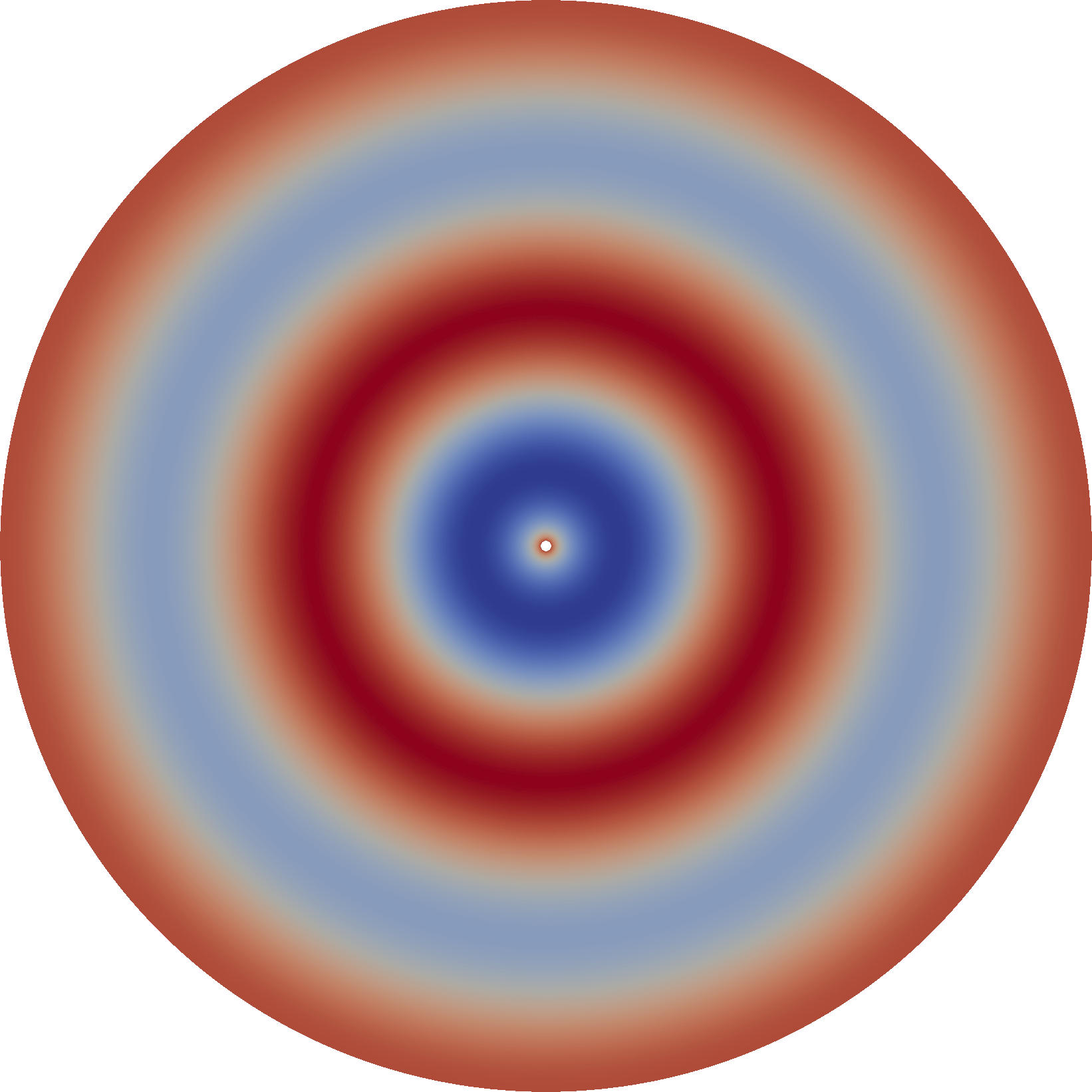} &
\includegraphics[width=3.9cm]{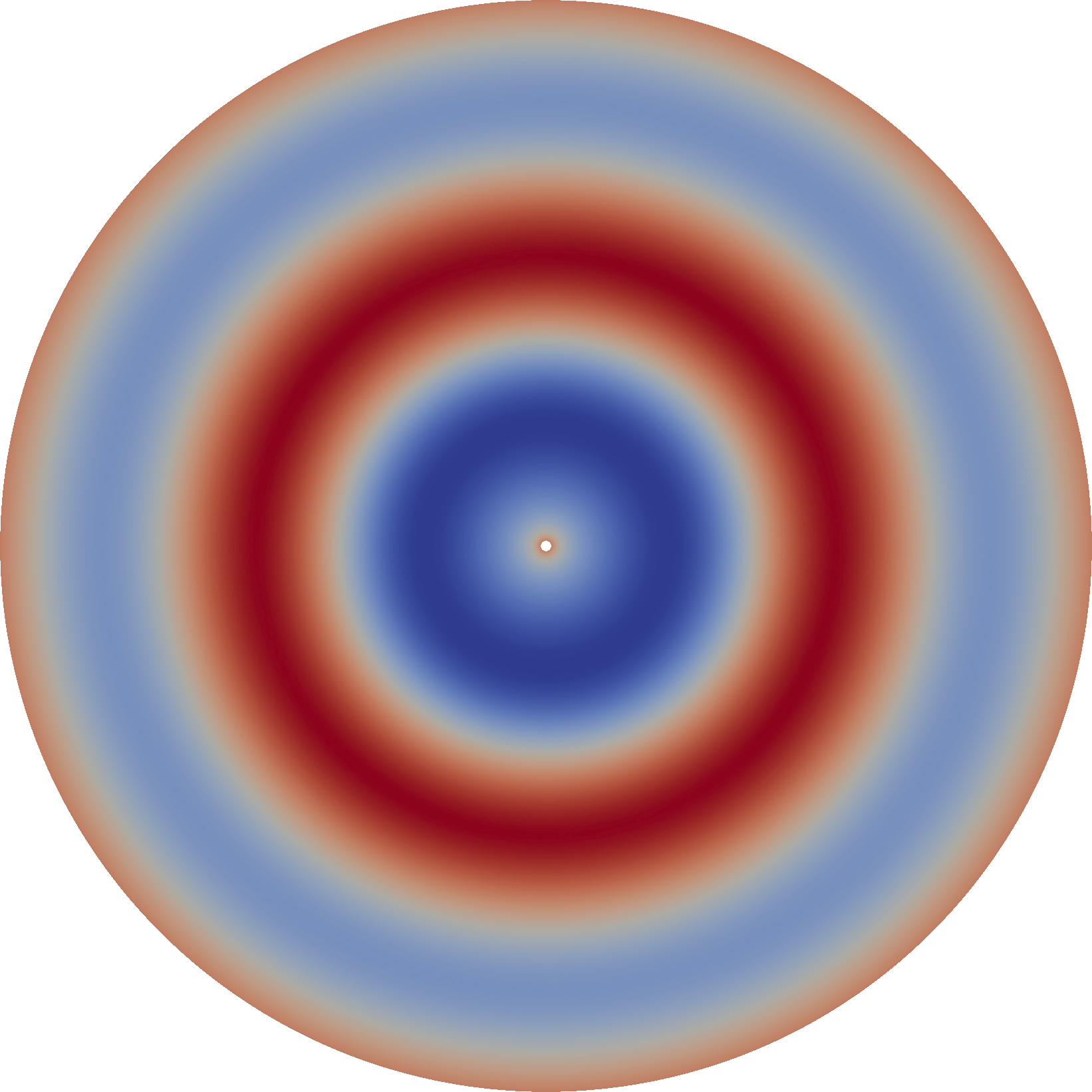} &
\includegraphics[width=3.9cm]{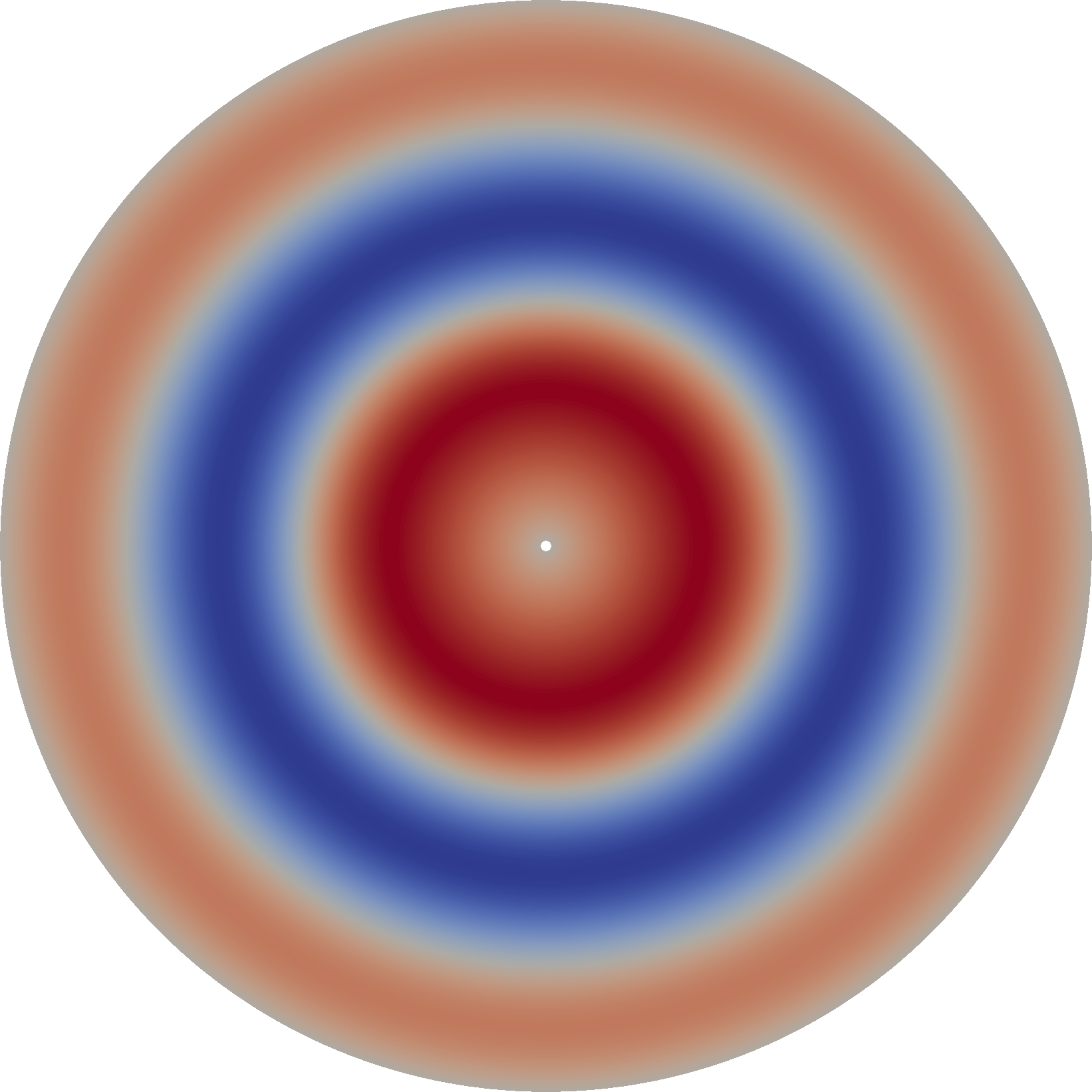} \\
\includegraphics[width=3.9cm]{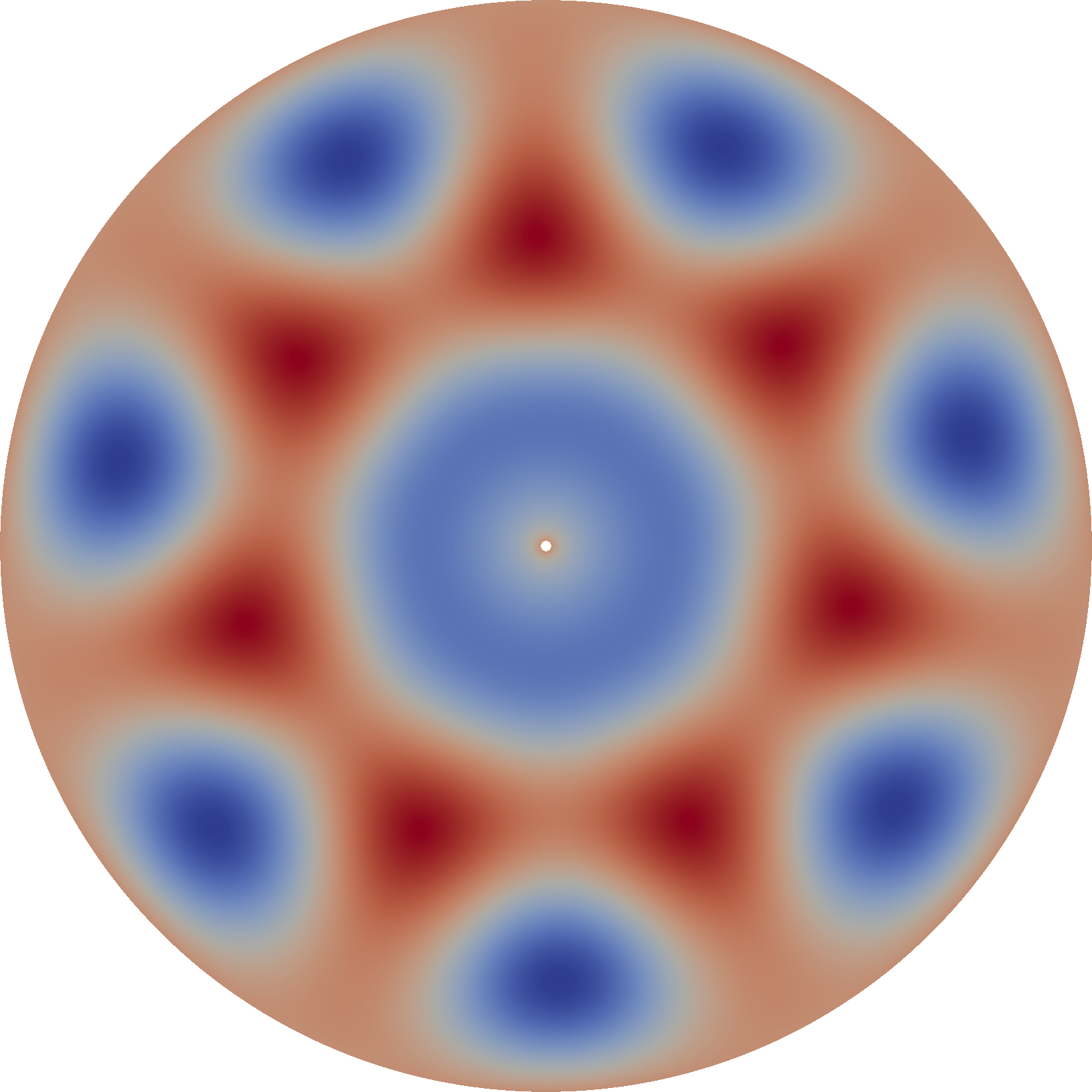} &
\includegraphics[width=3.9cm]{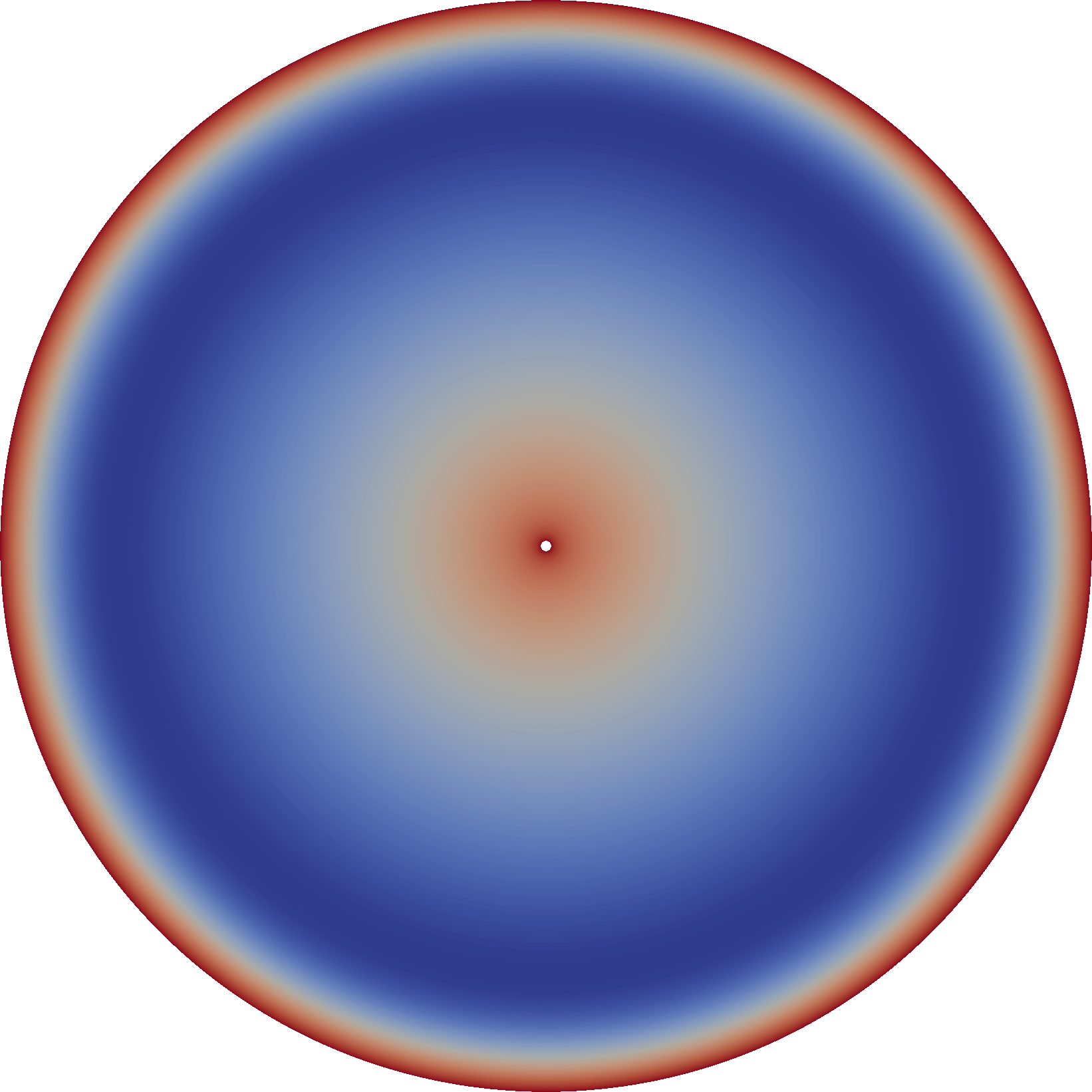} &
\includegraphics[width=3.9cm]{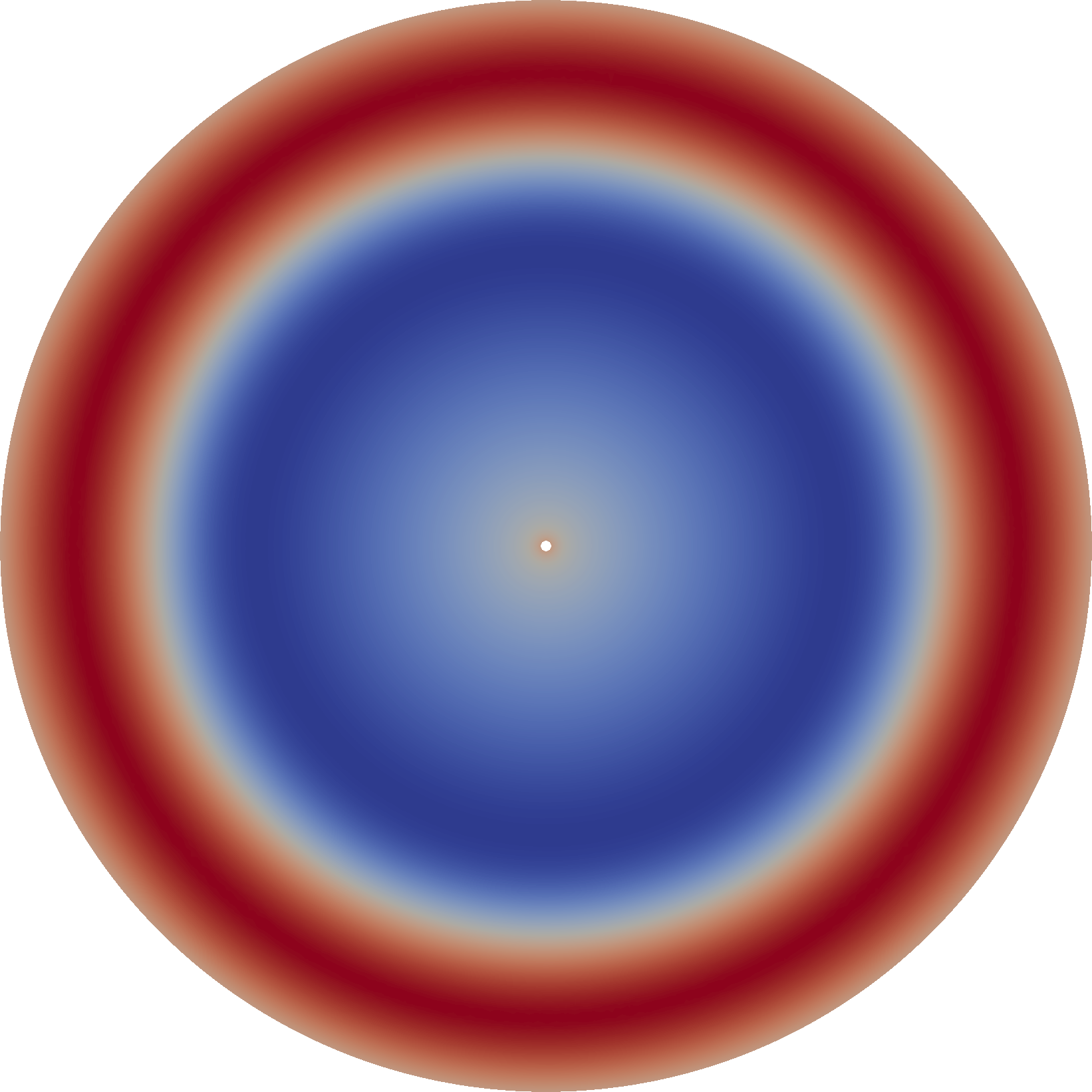} \\
\includegraphics[width=3.9cm]{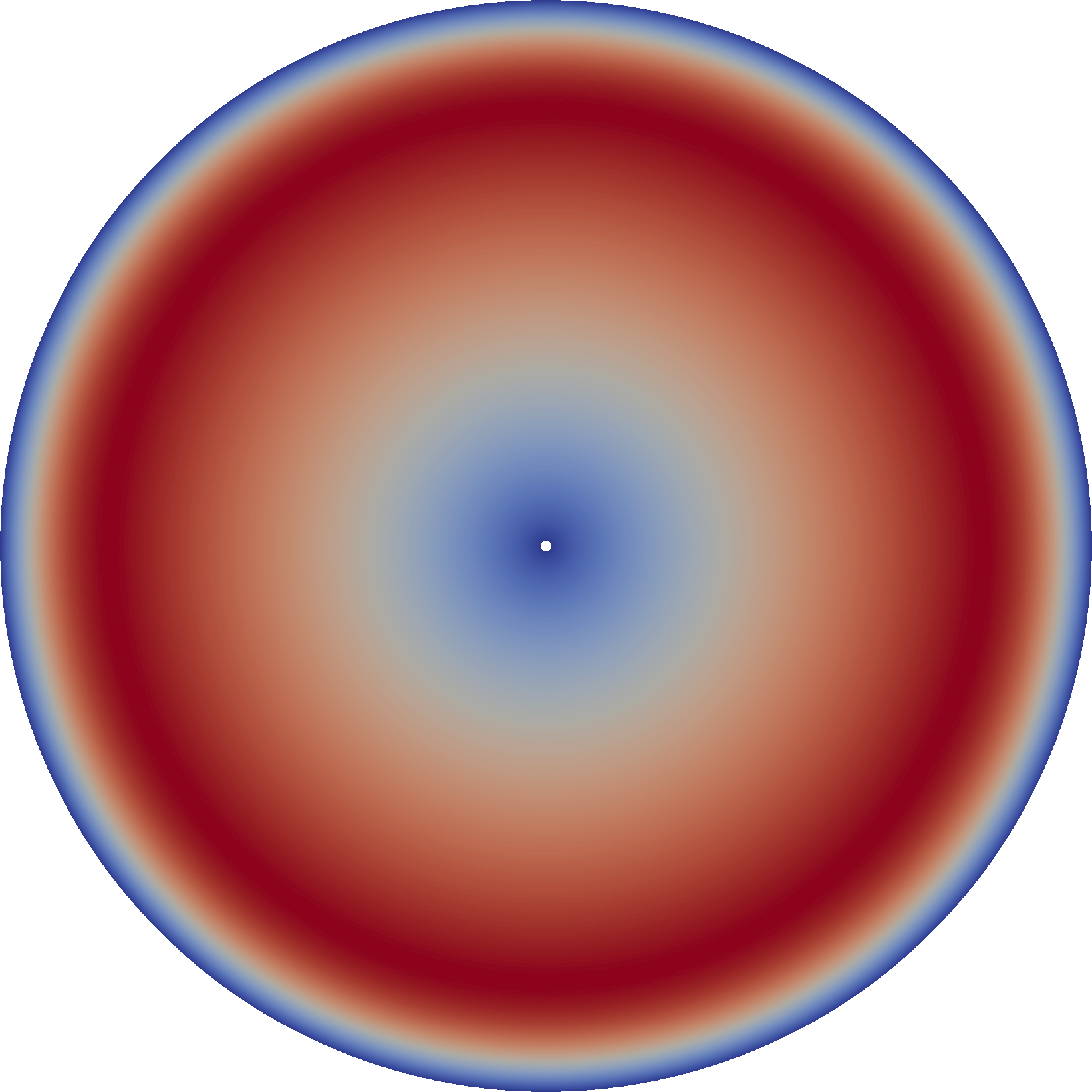} &
\includegraphics[width=3.9cm]{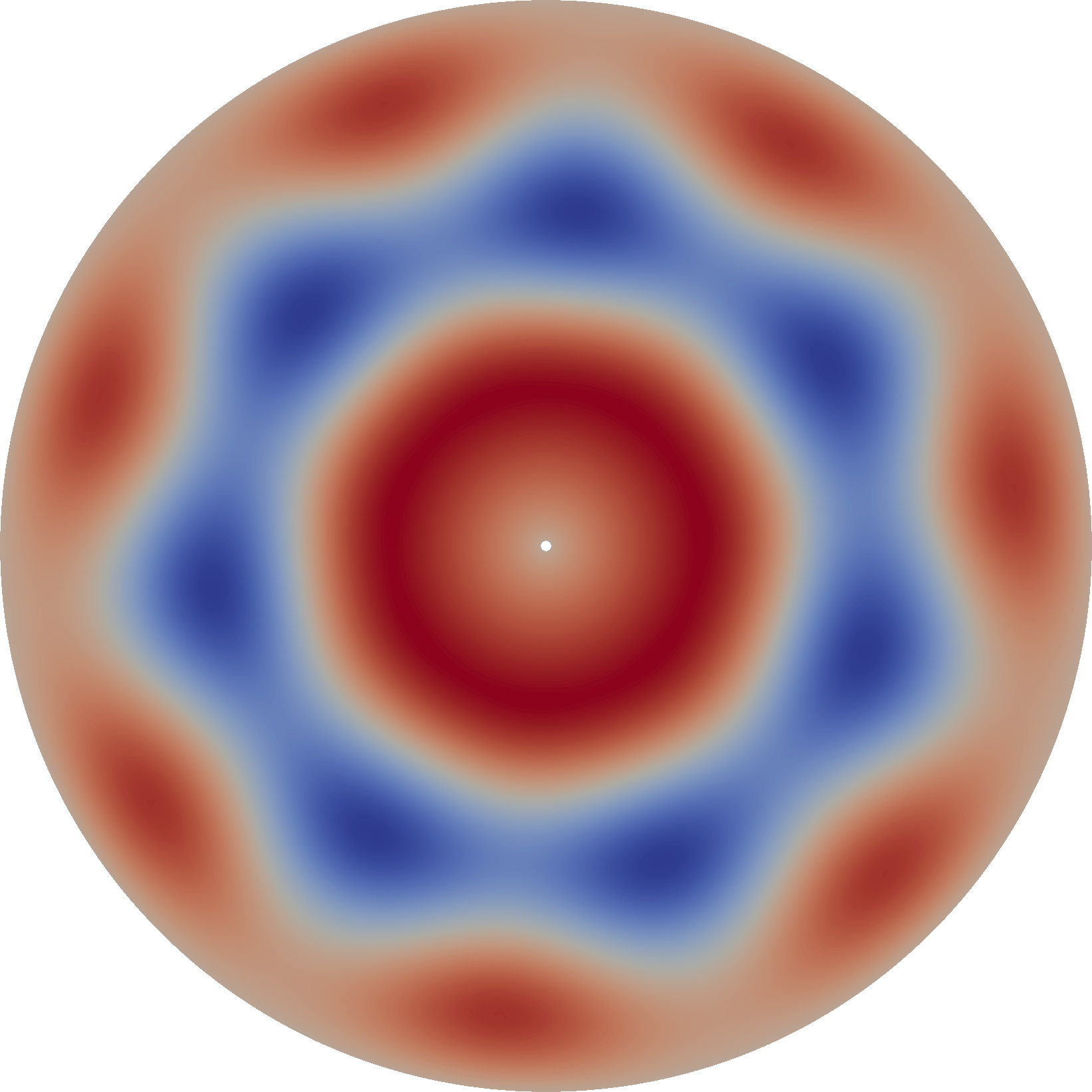} &
\includegraphics[width=3.9cm]{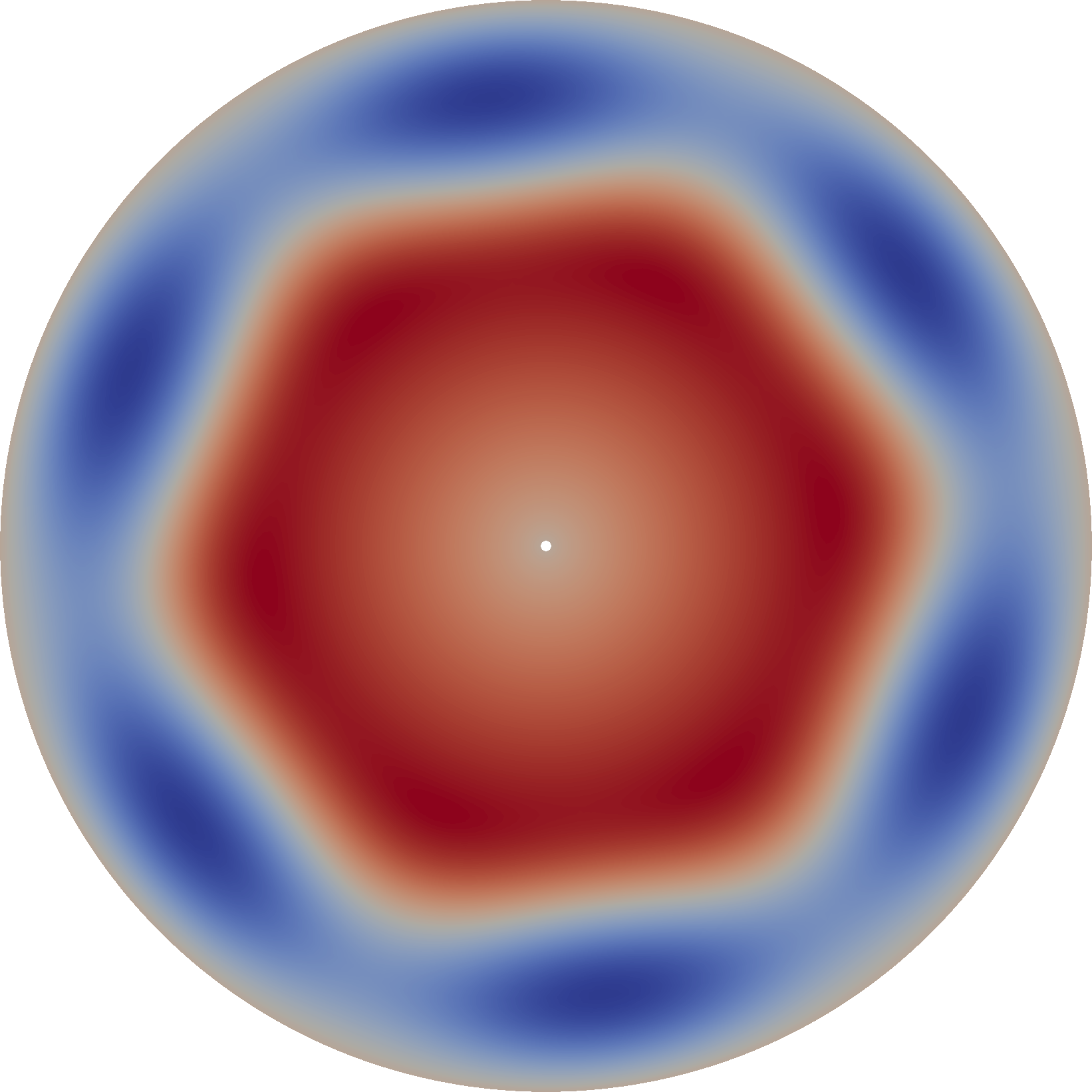} \\
\end{tabular}
\caption{Solutions of the Yamabe equation, found using deflation and negation. (All subfigures are plotted with different colour bars.)}
\label{fig:yamabe}
\end{figure}

The equation \eqref{eqn:yamabe} was discretised using standard linear finite elements via FEniCS.
An undamped Newton iteration was used to solve the nonlinear problem, with at most
100 iterations and an absolute residual termination tolerance of $10^{-10}$.  The mesh of 15968
vertices was generated with gmsh \cite{geuzaine2009}. Starting from the constant initial guess of 1,
deflation was successively applied with shift parameter $\alpha = 10^{-2}$.  After each deflation, the
nonlinear solver tolerance was reduced by a factor of $10^{-2}$. With this configuration, deflation
found 7 solutions, all starting from the same initial guess. Finally, further solutions were identified by
using the negative of each solution in turn as the initial guess for the (deflated) solver; the
negatives of solutions 4 and 6 yielded two additional distinct solutions. All of the 9 solutions identified with
this procedure are plotted in figure \ref{fig:yamabe}. The nonnegative solution was the
$7^{\mathrm{th}}$ to be identified. It is known that other solutions exist that this procedure has not found.

\begin{table}
\centering
\begin{tabular}{c|c}
\toprule
\# of deflations & average Krylov iterations per solve \\
\midrule
0 & 15.2 \\
1 & 17.1 \\
2 & 15.1 \\
3 & 16.9 \\
4 & 11.2 \\
5 & 12.4 \\
6 & 10.9 \\
7 & 15.5 \\
8 & 13.9 \\
\bottomrule
\end{tabular}
\caption{The performance of the preconditioning strategy of section \ref{sec:preconditioning}
on the Yamabe problem, section \ref{sec:yamabe}. As more solutions are deflated, the
number of Krylov iterations required does not significantly increase, indicating that the
preconditioning strategy suggested is effective.}
\label{tab:yamabe_preconditioning}
\end{table}

The linear systems arising in the Newton algorithm were solved using GMRES and the
GAMG classical algebraic multigrid algorithm via PETSc, with two
smoother iterations of Chebyshev and SOR. The relative and absolute Krylov solver tolerances were
both set to $10^{-12}$.  The average number of Krylov iterations required in the Newton iterations
for each solution are listed in table \ref{tab:yamabe_preconditioning}. The number of Krylov iterations
required for the deflated solves did not change significantly. This suggests that the
preconditioning strategy proposed in section \ref{sec:preconditioning} is effective, even as
several solutions are deflated. Additional runs were performed on ARCHER, the UK national supercomputer,
up to approximately two billion degrees of freedom; the deflated preconditioner was still effective,
even for these very fine discretisations.

While this example demonstrates the power of the deflation technique presented here, it also
emphasises its drawbacks: the algorithm is highly sensitive to the choice of the shift parameter.
Varying the shift parameter from $10^0, 10^{-1}, \dots, 10^{-7}$, the deflation procedure identifies
between 1 and 7 solutions. At present we are unable to give \emph{a priori} guidance on the choice of
the shift parameter, and reluctantly resort to numerical experimentation.

\subsection{Flow bifurcation: Navier--Stokes} \label{sec:navier_stokes}
The previous examples have demonstrated the utility of deflation for finding distinct solutions from the
same initial guess. In this example, we combine continuation (in Reynolds number) and deflation to
trace out the solution branches of the flow in a channel that undergoes a sudden expansion
\cite{benjamin1978a,fearn1990,sobey1993,drikakis1997,cliffe2014}.  Continuation and deflation are natural complements:
continuation uses the solutions identified for previous parameter values to trace out known solution
branches, whereas deflation attempts to find unknown solutions which may or may not lie on unknown
solution branches.

The system considered is the nondimensionalised steady incompressible Newtonian Navier--Stokes equations:
\begin{align} \label{eqn:ns}
-\frac{1}{\mathrm{Re}} \nabla^2 u + u \cdot \nabla u + \nabla p &= 0, \\
                                       \nabla \cdot u  &= 0,
\end{align}
where $u$ is the vector-valued velocity, $p$ is the scalar-valued pressure, and $\mathrm{Re}$ is the
Reynolds number.  The geometry is the union of two rectangles, $\Omega = (0, 2.5) \times (-1, 1)
\cup (2.5, 150) \times (-6, 6)$. Poiseuille flow is imposed at the inflow boundary on the left; an
outflow boundary condition ($\nabla u \cdot n = pn$, where $n$ is the unit outward normal) is imposed on the right; and no-slip ($u = 0$) imposed
on the remaining boundaries.  This configuration is symmetric around $y = 0$, and permits a symmetric solution.
It is well known (see e.g. \cite{cliffe2014}) that for low Reynolds numbers, this symmetric solution
is stable. At a critical Reynolds number, the system undergoes a pitchfork bifurcation, and the
system permits an unstable symmetric solution and (possibly several) pairs of asymmetric solutions.
The system \eqref{eqn:ns} is discretised using the Taylor--Hood finite element pair
\cite{taylor1973} via FEniCS; undamped Newton iteration is used for the resulting
nonlinear problems, with the undeflated linear systems solved on 16 cores with MUMPS
\cite{amestoy2001} via PETSc.

The approach to trace out the solution branches proceeds as follows. The continuation path starts at
$\mathrm{Re} = 10$ and uses a zero initial guess to find the first (symmetric) solution branch.
Then continuation steps of $\Delta \mathrm{Re} = 0.5$ are made. The initial guess for the Newton
iteration for \eqref{eqn:ns} is set to each solution identified for the previous value of
$\mathrm{Re}$ in turn; once each solution for the new parameter value is found, the residual is
deflated with $p = 1$, $\alpha = 1$. Newton's method is deemed to have failed if it has not converged within
20 iterations. Once all known solution
branches have been continued, the Newton iteration is initialised with the average of the available
previous solutions, and an attempt is made to locate any nearby unknown solutions via deflation.
We emphasise that more sophisticated continuation algorithms such as pseudo-arclength continuation
\cite{keller1977,allgower1993} could naturally be combined with deflation.

\begin{figure}
\begin{tabular}{c}
\includegraphics[width=12.5cm]{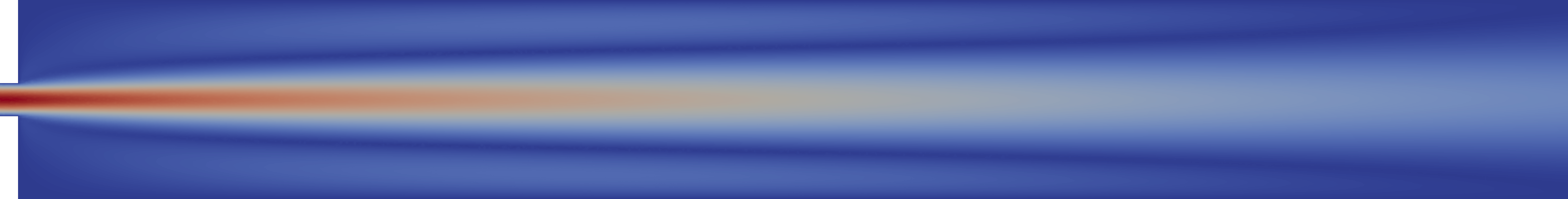} \\
\includegraphics[width=12.5cm]{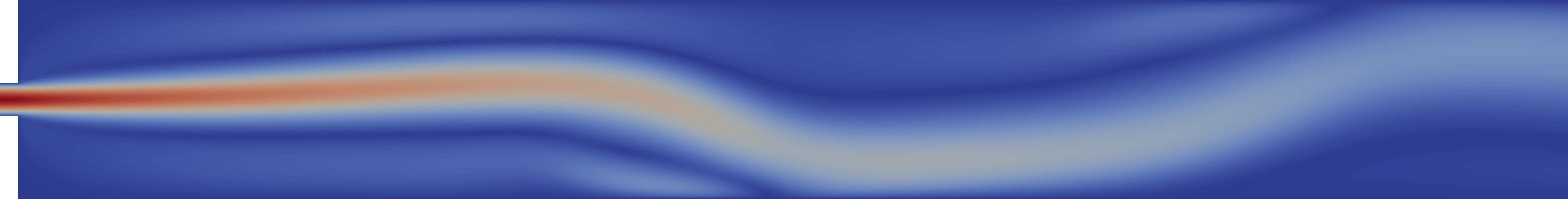} \\
\includegraphics[width=12.5cm]{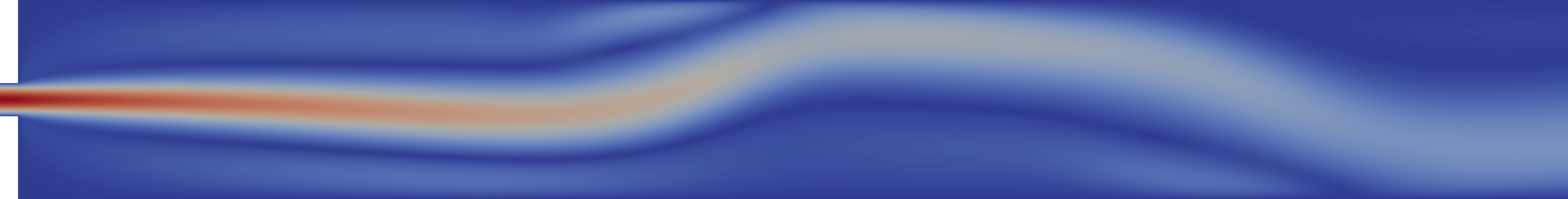} \\
\includegraphics[width=12.5cm]{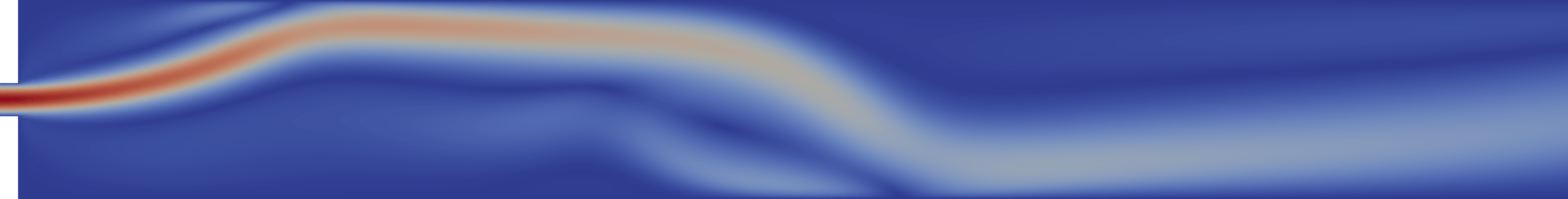} \\
\includegraphics[width=12.5cm]{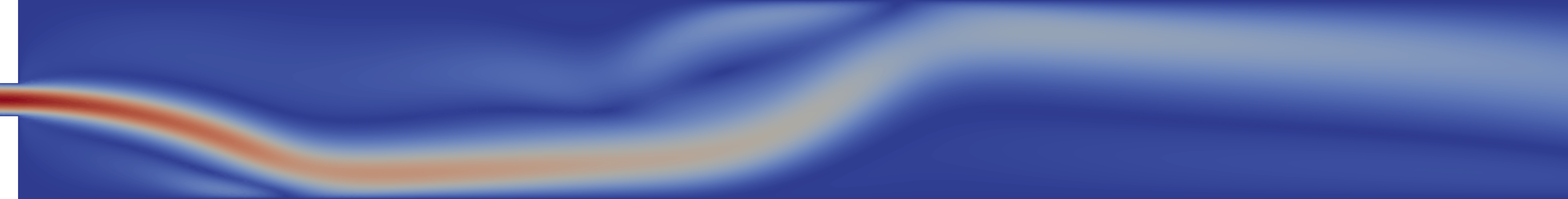} \\
\includegraphics[width=12.5cm]{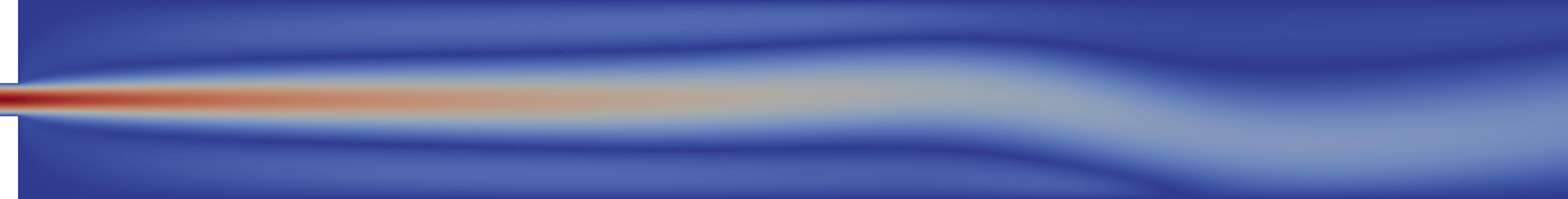} \\
\includegraphics[width=12.5cm]{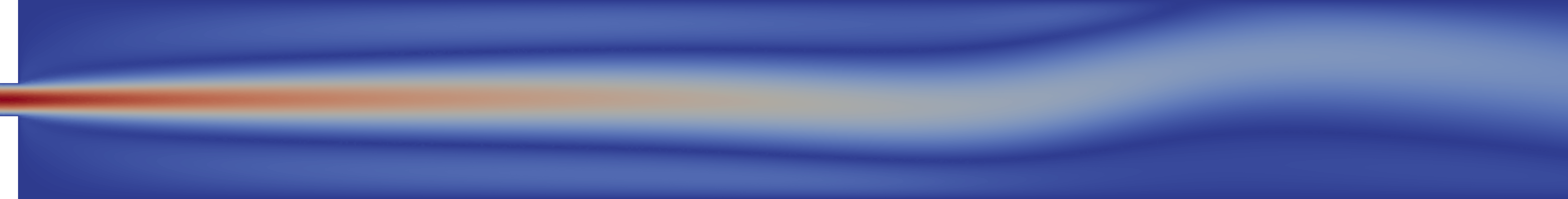} \\
\end{tabular}
\caption{The velocity magnitude of the solutions of the Navier--Stokes sudden expansion problem, at $\mathrm{Re} = 100$, found using deflation and reflection.}
\label{fig:navier_stokes}
\end{figure}

The algorithm was terminated at $\mathrm{Re} = 100$, having found 6 solutions (figure
\ref{fig:navier_stokes}). A bifurcation diagram is shown in figure \ref{fig:navier_stokes_avg_bifurcation},
with functional
\begin{equation}
J(u, p) = \pm \int_{\Omega} \left|u - \mathcal{R}u\right|^2,
\end{equation}
where $\mathcal{R}$ is the reflection operator through $y = 0$ and the sign is chosen according
to whether the jet initially veers up or down. This functional measures the asymmetry of the solution.
As we expect one symmetric solution and pairs of asymmetric solutions,
clearly the procedure has not found all possible solutions. This is confirmed by initialising a Newton
iteration with the reflection of each asymmetric solution found; this identifies one
extra solution (the second-last of figure \ref{fig:navier_stokes}).

\begin{figure}
\centering
  \begin{subfigure}[b]{0.95\textwidth}
  \includegraphics[width=\textwidth]{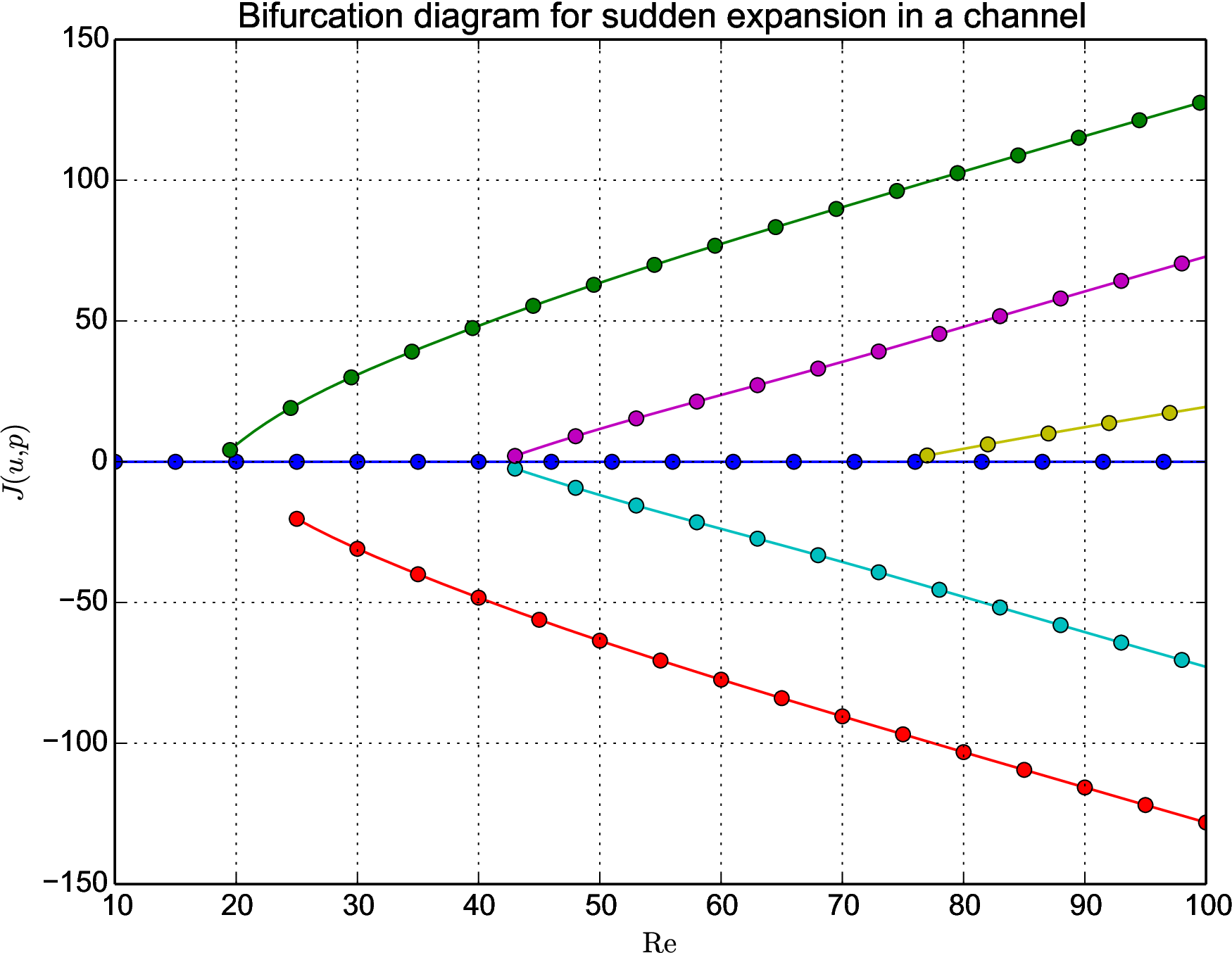}
  \caption{without reflection}
  \label{fig:navier_stokes_avg_bifurcation}
  \end{subfigure}
  \begin{subfigure}[b]{0.95\textwidth}
  \includegraphics[width=\textwidth]{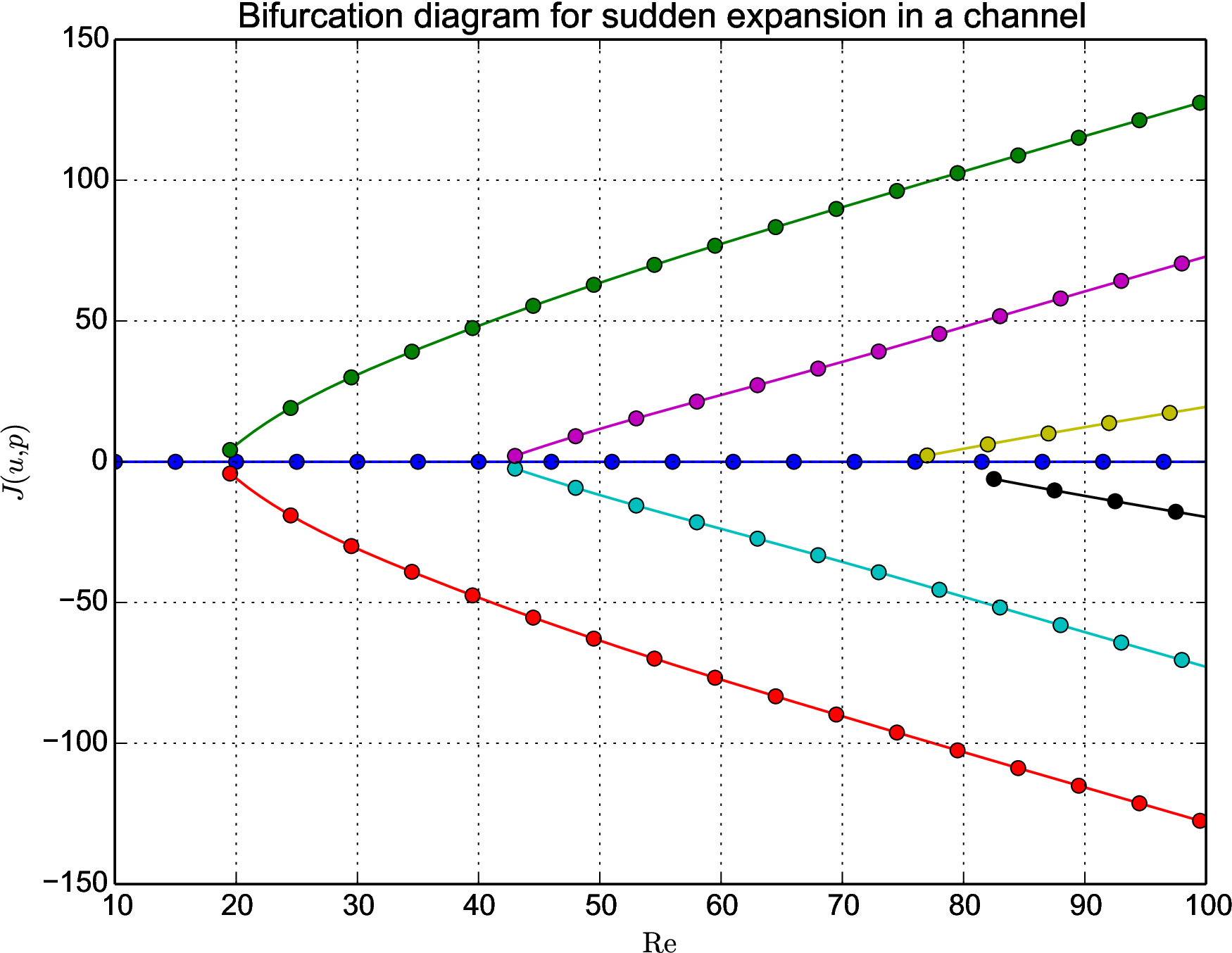}
  \caption{with reflection}
  \label{fig:navier_stokes_rft_bifurcation}
  \end{subfigure}
\caption{Bifurcation diagrams for the Navier--Stokes example of section
\ref{sec:navier_stokes}. Top: Only deflation was employed in the continuation
algorithm (in particular, reflection was not used). The algorithm has missed
one branch of seven. Bottom: When reflection is used in addition to generate initial guesses from known solutions,
the algorithm finds the branch that was overlooked.}
\end{figure}

This motivates the inclusion of reflection in the continuation algorithm itself.
Since for all Reynolds numbers we expect to have an odd number of solutions, the
algorithm was modified as follows: if an even number of solutions has been found
after a continuation step, use the reflection of each known solution as the
initial guess for Newton's method. With this modification, all seven solutions
were found by $\mathrm{Re} = 82.5$. The improved bifurcation diagram with the
addition of reflection is plotted in figure
\ref{fig:navier_stokes_rft_bifurcation}.  By including more knowledge of the
problem in the generation of new initial guesses from known solutions, the
continuation algorithm finds the additional branch that the previous approach
missed.

\section{Conclusion}

We have presented a deflation algorithm for computing distinct solutions of systems of partial
differential equations. By systematically modifying the residual of the problem, known solutions
may be eliminated from consideration and new solutions sought.
While deflation does not guarantee that all solutions of the problem will be found, it has proven itself
useful on several nonlinear PDEs drawn from special functions, phase separation, differential
geometry and fluid mechanics. It provides a powerful complement to continuation methods, and
motivates further developments of sophisticated nonlinear solvers \cite{brune2013}.

A possible future application of deflation is to the computation of multiple local optima of
PDE-constrained optimisation problems, possibly with additional equality constraints
\cite{hinze2009,borzi2011}, by means of computing distinct solutions of the associated
Karush--Kuhn--Tucker equations.

\bibliographystyle{siam}
\bibliography{literature}

\end{document}